\documentclass[english, 11pt]{article}

\usepackage[T1]{fontenc}
\usepackage[latin9]{inputenc}
\usepackage{amsthm}
\usepackage{amsmath}
\usepackage{amssymb}
\usepackage{color}
\usepackage{esint}
\usepackage{verbatim}

\makeatletter
\numberwithin{equation}{section}
\numberwithin{figure}{section}
\theoremstyle{plain}
\newtheorem{thm}{\protect\theoremname}
\theoremstyle{definition}
\newtheorem{defn}[thm]{\protect\definitionname}
\theoremstyle{remark}
\newtheorem{rem}[thm]{\protect\remarkname}
\theoremstyle{plain}

\newtheorem{prop}{Proposition}
\newtheorem{lemma}[thm]{Lemma}

\@ifundefined{date}{}{\date{}}
\makeatother

\usepackage{babel}
  \providecommand{\corollaryname}{Corollary}
  \providecommand{\definitionname}{Definition}
\providecommand{\remarkname}{Remark}
\providecommand{\theoremname}{Theorem}

\begin{document}

\title{High Mach number limit for Korteweg fluids with density dependent viscosity}

\author{Matteo Caggio{*}, Donatella Donatelli{*}}

\date{{*} Department of Information Engineering, Computer Science and Mathematics, University of L'Aquila, Italy \\ \smallskip \small matteo.caggio@univaq.it, donatella.donatelli@univaq.it}

\maketitle
\begin{abstract}
	The aim of this paper is to investigate the regime of high Mach number flows for compressible barotropic fluids of Korteweg type with density dependent viscosity. In particular we consider the  models for  isothermal capillary and quantum compressible fluids. For the capillary case we prove the existence of weak solutions and related properties for the system without pressure, and the convergence of the solution in the high Mach number limit. This latter is proved also in the quantum case for which a weak-strong uniqueness analysis is also discussed in the framework of the so-called "augmented" version of the system. Moreover,  as byproduct of our results,  in the case of a capillary fluid with a  special choice of the  initial velocity datum, we obtain an interesting property concerning the  propagation of vacuum zones.
\end{abstract}
\smallskip


\tableofcontents{}

\newpage{}

\section{Introduction} \label{intro}

We consider a compressible Navier--Stokes system for a barotropic fluid with density dependent viscosity in a periodic domain $\Omega=\mathbb{T}^{d}
$ ($d=2$ or 3)
\begin{equation} \label{cont.}
\partial_{t}\varrho+\textrm{div}\left(\varrho u\right)=0,
\end{equation}
\begin{equation} \label{mom.}
\partial_{t}\left(\varrho u\right)+\textrm{div}\left(\varrho u\otimes u\right)-\textrm{div}\left(2\mu\left(\varrho\right)D\left(u\right)\right)-\nabla\left(\lambda\left(\varrho\right)\textrm{div}u\right)+\frac{\nabla p(\varrho)}{\mathcal{M}a^{2}}=f,
\end{equation}
with initial data
$$\left.\left(\varrho,u\right)\right|_{t=0}=\left(\varrho_{0},u_{0}\right).$$ 
Here, $u=u(t,x)$ and $\varrho=\varrho(t,x)$ stand for the velocity field and density of the fluid, functions of the spatial position $x\in\mathbb{R}^{3}$ and time $t$. The quantity $D\left(u\right)=(\nabla u+\nabla^{t}u)/2
$ represents the strain tensor and $f$ is a given external forcing that will be specified later. We denote by $\lambda\left(\varrho\right)$ and $\mu\left(\varrho\right)$ the two viscosity coefficients, functions of the fluid density. We consider a pressure $p$ of the type $p\left(\varrho\right)=a\varrho^{\gamma}$ where $a>0$ and $\gamma>1$ are physical constant quantities. The range of $\gamma$ will be specified later.

The above system is introduced in its non-dimensional form and it contains the Mach number $\mathcal{M}a$, while the other fluid mechanics numbers are set equal to one. 

The Mach number, physically, is given by the ratio of the reference fluid velocity and the sound speed. In general the value of the Mach number depends on the conditions of the physical phenomenon under consideration  and it is related to the compressibility of the fluid according to its low or high values. In the case of a low Mach number regime the speed of sound tends to infinity while the pressure becomes almost constant and doesn't generates density variations, as a consequence compressibility can be ignored and the final asymptotical  physical state  is the incompressible one. Conversely  if  the fluid speed increases beyond the sound speed (as in the so called supersonic flows), then the Mach number is high and the compressibility effects have to be taken into consideration.

Since  in  many real world phenomena, such as ocean flows, astrophysics flows, fluid flows in engineering devices, the fluid velocities are smaller compared to the sound speed, the study of the low Mach number regime has been of large interest.  In the literature a large number of authors have dealt with this case for different types of fluids, see for instance \cite{CN17}, \cite{CDM13}, \cite{DG99}, \cite{DFN10},  \cite{DM12a}, \cite{DM16}, \cite{DBN18}, \cite{Donatelli-Trivisa-2008}, \cite{L-P.L.M98}  and references therein.

On the other hand, although the high Mach number  values are of great importance in the dynamics of aircrafts, this regime and its related limit analysis turns out to be less studied in comparison to the opposite asymptotic limit. 

Previous results in this context have been obtained by Haspot \cite{BH-1} and Liang \cite{L}. For a pressure behaving like a power law, namely $p(\varrho)=\varrho^{\gamma}$, with $\gamma>1$, Liang \cite{L} proved the convergence to the pressureless system for the following form of the stress tensor and more general form of viscosity coefficients 
\begin{equation} \label{stress}
S\left(u\right)=\mu\left(\varrho\right)\nabla u+\lambda\left(\varrho\right)\textrm{div}u
\end{equation}
\begin{equation} \label{mu_l}
\mu\left(\varrho\right)>0, \ \ \mu\left(\varrho\right)+d\lambda\left(\varrho\right)\geq0
\end{equation}
\begin{equation} \label{alpha}
\mu\left(\varrho\right)=\nu\varrho^{\alpha}, \ \ \lambda\left(\varrho\right)=2\left(\alpha-1\right)\nu\varrho^{\alpha}, \ \ \alpha\geq1.
\end{equation}
As remarked by Liang \cite{L}, more general viscosities are allowed at the cost of the stress tensor having the form (\ref{stress}). Moreover, Liang \cite{L} also remarked that for $\alpha=1$ (and also $\alpha<1$ with slight modification of the arguments) the convergence holds for the symmetric stress tensor $\textrm{div}\left(\varrho\frac{\nabla u+\nabla^{t}u}{2}\right)$ where the existence of a global weak solutions for the system (\ref{cont.}) - (\ref{mom.}) is proved in Li and Xin \cite{LX} (see also Vasseur and Yu \cite{VY}). A recent extension of these results for non-linear density dependent viscosities have been obtained by Bresch et al. \cite{BD-5}. Note that, in \cite{L}, the author assumed the existence of the global weak solution for the pressureless system in the spirit of Haspot \cite{BH-1}. Moreover, under some assumptions on the initial velocity field and for $\alpha\geq3/2$, Liang \cite{L} proved a rate of convergence for the density field in a suitable Lebesgue norm, while Haspot and Zatorska \cite{HZ} obtain a rate of convergence for the density for the one--dimensional Cauchy problem for $1<\alpha\leq3/2$. In a similar framework, Haspot \cite{BH-1} proved the high compressible limit in the sense of distribution and discussed the global existence for the pressureless system. In particular, the author constructed explicit global weak solutions where $\varrho$ is a solution of a porous media or heat equation according to the different values of $\alpha$ in (\ref{alpha}) and with irrotational initial velocity expressed as a gradient of a given potential, function of the initial density.
%

In this paper we consider  the following high Mach number regime, namely the case  when 
$$\mathcal{M}a=\varepsilon^{-1/2}, \qquad \varepsilon\to 0 $$
 in the context of capillary and quantum fluids, namely we will deal with the following Korteweg system of equations

\begin{equation} \label{cont._1}
\partial_{t}\varrho+\textrm{div}\left(\varrho u\right)=0,
\end{equation}
\[
\partial_{t}\left(\varrho u\right)+\textrm{div}\left(\varrho u\otimes u\right)-2\nu\textrm{div}\left(\varrho D\left(u\right)\right)+\varepsilon\nabla p(\varrho)
\]
\begin{equation} \label{mom._1}=\varrho\nabla\left(\kappa\left(\varrho\right)\Delta\varrho+\frac{1}{2}\kappa^{\prime}\left(\varrho\right)\left|\nabla\varrho\right|^{2}\right),
\end{equation}
with initial data
$$\left.\left(\varrho,u\right)\right|_{t=0}=\left(\varrho_{0},u_{0}\right),
$$ 
where we set $\mu\left(\varrho\right)=\nu\varrho$ with $\nu$ constant viscosity coefficients and $\lambda\left(\varrho\right)=0$. Here, the right-hand-side of (\ref{mom._1}) represents the so-called Korteweg tensor with $\kappa=\kappa(\varrho)$ surface tension. The analysis is motivated by a recent attention given to the system (\ref{cont._1}) - (\ref{mom._1}) (see for example Antonelli and Spirito \cite{AS_3}, \cite{AS_2}, \cite{AS_1}, \cite{AS}, and Bresch et al. \cite{BGL}, \cite{BD-5}).  

For a constant value of $\kappa$, namely $\kappa(\varrho)=\kappa$, the above system describes the motion of a capillary fluids for which the Korteweg tensor becomes
$$\varrho\nabla_{x}\left(\kappa\left(\varrho\right)\Delta_{x}\varrho+\frac{1}{2}\kappa^{\prime}\left(\varrho\right)\left|\nabla_{x}\varrho\right|^{2}\right)=\kappa\varrho\nabla\Delta\varrho,$$
while the choice $\kappa(\varrho)=\kappa^2/\varrho$ leads to the so-called quantum Bohm identity for which we deal with quantum fluids, $$\varrho\nabla_{x}\left(\kappa\left(\varrho\right)\Delta_{x}\varrho+\frac{1}{2}\kappa^{\prime}\left(\varrho\right)\left|\nabla_{x}\varrho\right|^{2}\right)=2\kappa^2\varrho\nabla\left(\frac{\Delta\sqrt{\varrho}}{\sqrt{\varrho}}\right).$$
For a more detailed discussion on Korteweg models for capillary and quantum fluids, the reader can refer to Benzoni-Gavage \cite{BG}.



For $\varepsilon=0$, the system (\ref{cont._1}) and (\ref{mom._1}) reduces to the  pressureless Navier-Stokes equations of Korteweg type. In the following, our analysis will focus on this system with $\kappa(\rho)=\kappa$ or $\kappa(\rho)=\kappa^{2}/\rho$. 

For capillary fluids, we will study the convergence of the  limit as $\varepsilon\rightarrow0$ to the weak solutions for the pressureless system, moreover we investigate also the existence of weak solutions for the pressureless system  and  some properties connected with the heat equation for a proper choice of the initial data. It is worth to point out that as a byproduct of this analysis we get an interesting result concerning the vacuum states of the density. Indeed in Theorem \ref{heat_eq} we will be able to prove that  for an initial density  $\rho_{0}$ strictly positive and an irrotational initial velocity, $u_{0}=\nabla\phi(\rho_{0})$  vacuum is not allowed for any $t>0$. In other words, in this special case, under the high compressibility effects and when the pressure effects are not relevant, then vacuum states may not appear in the non vacuum regions $(\rho_{0}>0)$. 

For quantum fluids the convergence in the limit of $\varepsilon\rightarrow0$ will be  discussed and a weak-strong uniqueness analysis will be performed. Weak-strong uniqueness means that a weak and strong solution emanating from the same initial data coincide as long as the latter exists. We perform a weak strong-uniqueness analysis for the pressureless system assuming the existence of the strong solution. This last analysis will follow the recent result of Bresch et al. \cite{BGL}  in which an ``augmented'' version of the quantum Navier-Stokes system combined with a relative energy inequality approach (see for example Feireisl et al. \cite{FSN}, \cite{FJN}) is used.




We would like to mention that, without the presence of the Korteweg tensor, the study of the high compressible limit in the context of density dependent viscosity fluids is related to the lack of compactness of the density. More precisely, in the case of constant viscosities it appears impossible to pass to the limit for $\varepsilon\rightarrow0$ because the $L^{\gamma}$-bound for the density coming from the pressure is no longer conserved. Only the $L^{1}$-bound related to mass conservation is preserved, and is not sufficient to pass to the limit since it does not provide enough compactness information (see Haspot \cite{BH-1} and Liang \cite{L}). 



This manuscript is organized as follows. In Section 2 we investigate the high compressible limit for capillary fluids and we study some properties of the related pressureless system. Section 3 is devoted to the analysis of the quantum case. Finally, we end the paper with some concluding remarks in Section 4 and an Appendix where a stability result for a the pressureless system in the quantum case is shown.

\section{Capillary fluids}
In this section we deal with the capillary case, namely when $\kappa\left(\varrho\right)=\kappa$ in (\ref{mom._1}), then the system (\ref{cont._1}) - (\ref{mom._1}) reads as
\begin{equation} \label{cont._1_cap}
\partial_{t}\varrho+\textrm{div}\left(\varrho u\right)=0,
\end{equation}
\begin{equation} \label{mom._1_cap}
\partial_{t}\left(\varrho u\right)+\textrm{div}\left(\varrho u\otimes u\right)-2\nu\textrm{div}\left(\varrho D\left(u\right)\right)+\varepsilon\nabla p(\varrho)=\kappa\varrho\nabla\Delta\varrho,
\end{equation}
with initial data
$$\left.\left(\varrho,u\right)\right|_{t=0}=\left(\varrho_{0},u_{0}\right).$$ 

For fixed $\varepsilon>0$, the global existence of weak solutions without smallness assumption on the data was studied by Bresch et al. \cite{BDL} in a periodic domain $\Omega=\mathbb{T}^{d}$ ($d=2,3$) under the following assumptions on the pressure field

\begin{equation} \label{press.}
p(\varrho)\geq0, \ \ p^{\prime}(\varrho)\geq0, \ \ \textrm{and} \ \ \Gamma\left(\varrho\right)\leq A\varrho^{\eta}\Pi\left(\varrho\right) \ \ \textrm{\ for \ large \ enough} \ \ \varrho,       
\end{equation}
where
\begin{equation} \label{press._1}
\Pi\left(\varrho\right)=\varrho\int_{\overline{\varrho}}^{\varrho}\frac{p\left(\tau\right)}{\tau^2}d\tau, \ \ \Gamma\left(\varrho\right)=\int_{\overline{\varrho}}^{\varrho}\tau p^{\prime}\left(\tau\right)d\tau,
\end{equation}
with $\overline{\varrho}$ constant reference density, $A$ positive constant and $\eta<+\infty$ when $d=2$ and $\eta<4$ when $d=3$. Note that, the assumption on large densities is satisfied in particular when the pressure behaves like a power law at infinity. In particular, we have
\begin{equation} \label{pi}
	\Pi\left(\varrho\right)=\varrho\int_{0}^{\varrho}\frac{p\left(\tau\right)}{\tau^2}d\tau=\frac{\varrho^{\gamma}}{\gamma-1}, \ \ 
	\Gamma\left(\varrho\right)=\int_{0}^{\varrho}\tau p^{\prime}\left(\tau\right)d\tau=\frac{\gamma}{\gamma+1}\varrho^{\gamma+1},
\end{equation}
where we have set $\overline{\varrho}=0$. More precisely, Bresch et al. \cite{BDL} proved the existence of weak solutions in the following class of regularity
\begin{equation} \label{weak_sol_reg}
\left\{ \begin{array}{c}
\varrho\in L^{2}\left(0,T;H^{2}\left(\Omega\right)\right)\cap L^{\infty}\left(0,T;H^{1}\left(\Omega\right)\right),\\
\nabla\sqrt{\varrho}, \ \sqrt{\varrho}u\in L^{\infty}\left(0,T;L^{2}\left(\Omega\right)\right),\\
\end{array}\right.
\end{equation}
satisfying the energy inequality
$$
\sup_{t\in(0,T)}\int_{\Omega}\frac{1}{2}\left(\left|\Lambda\right|^{2}
+\varepsilon\Pi\left(\varrho\right)
+\kappa\left|\nabla\varrho\right|^{2}\right)dx+2\nu\int_0^T\int_{\Omega}\left|\mathcal{S} \right|^2dxdt
$$
\begin{equation} \label{ei_weak}
\leq \int_{\Omega}\frac{1}{2}\left(\varrho_{0}\left|u_{0}\right|^{2}+\varepsilon\Pi\left(\varrho_{0}\right)+\kappa\left|\nabla\varrho_{0}\right|^{2}\right)dx
\end{equation}
with $\Lambda$ such that $\varrho u=\sqrt{\varrho}\Lambda$, and $\mathcal{S}\in L^2((0,T)\times\Omega)$ such that $\sqrt{\varrho}\mathcal{S}=\mbox{Symm}(\nabla(\varrho u)-2\nabla\sqrt{\varrho}\otimes\sqrt{\varrho} u)$ in $\mathcal{D}'$,
and the so-called Bresch-Desjardins entropy inequality (see Bresch et al. \cite{BDL})
$$
\sup_{t\in(0,T)}\int_{\Omega} \left(\frac{1}{2}\left|\Lambda+\nu\nabla\sqrt{\varrho}\right|^{2}
+\varepsilon\Pi(\varrho)
+\frac{\kappa}{2}\left|\nabla\varrho\right|^{2}\right)dx
$$
$$
+4\nu\int_0^T\int_{\Omega} \varepsilon p'(\varrho) \left| \nabla \sqrt{\varrho} \right|^2 dxdt
+\nu\kappa\int_0^T\int_{\Omega}\left|\nabla\nabla\varrho\right|^{2}dxdt
$$
$$
+2\nu\int_0^T\int_{\Omega}\left| \mathcal{A}\right|^2 dxdt
$$
\begin{equation} \label{ineq.}
\leq \int_{\Omega}\frac{1}{2}\left(\left|\sqrt{\varrho_{0}}u_{0}+
\nu\nabla\sqrt{\varrho_{0}}\right|^{2}+\kappa\left|\nabla\varrho_{0}\right|^{2}\right)dx,
\end{equation}
with $\mathcal{A}\in L^2((0,T)\times\Omega)$ such that $\sqrt{\varrho}\mathcal{A}=\mbox{Asymm}(\nabla(\varrho u)-2\nabla\sqrt{\varrho}\otimes\sqrt{\varrho} u)$ in $\mathcal{D}'$,
with the following conditions on the initial data
\begin{equation} \label{in_en}
\int_{\Omega}\frac{1}{2}\left(\varrho_{0}\left|u_{0}\right|^{2}+\varepsilon\Pi\left(\varrho_{0}\right)+\kappa\left|\nabla\varrho_{0}\right|^{2}\right)dx<+\infty,
\end{equation}
\begin{equation} \label{f_BDL}
\frac{1}{2}\int_{\Omega}\varrho_{0}\left|\nu\nabla\log\varrho_{0}\right|^{2}<+\infty.
\end{equation}
They proved the following Theorem.
\begin{thm} \label{BDJ}
	Let $d=2$ or $3$. Then, there exists a global weak solution $(\varrho,u)$ of (\ref{cont._1_cap}) and (\ref{mom._1_cap}), that means a solution satisfying (\ref{weak_sol_reg}) - (\ref{f_BDL}), with the continuity equation satisfying the following weak formulation for all $\varphi\in C_{c}^{\infty}\left(\left[0,T\right];C^{\infty}\left(\Omega\right)\right)
	$ such that $\varphi(T,\cdot)=0$,
	\begin{equation} \label{cont_BDL}
	\int_{\Omega}\varrho_{0}\cdot\varphi\left(0,\cdot\right)dx+\int_{0}^{T}\int_{\Omega}\varrho\cdot\partial_{t}\varphi+\int_{0}^{T}\int_{\Omega}\left(\sqrt{\varrho}\sqrt{\varrho}u\right):\nabla\varphi dxdt=0;
	\end{equation}
	and the momentum equations satisfying the following weak formulation for all $\varphi\in C_{c}^{\infty}\left(\left[0,T\right];C^{\infty}\left(\Omega\right)\right)
	$ such that $\varphi(T,\cdot)=0$,
	$$\int_{\Omega}\varrho_{0}u_{0}\cdot\varrho_{0}\varphi\left(0,\cdot\right)dx+\int_{0}^{T}\int_{\Omega}\left[\varrho^{2}u\cdot\partial_{t}\varphi+\frac{3}{2}\varrho u\otimes\varrho u:\nabla\varphi\right.
$$
$$
+\sqrt{\varrho} u \otimes \sqrt{\varrho} u : \varphi \nabla \varrho 
+\varepsilon\Gamma\left(\varrho\right)\textrm{div}\varphi
$$
\begin{equation} \label{wf_mom._BDL}
\left. -2\nu \varrho D\left(u\right) \cdot \nabla(\varrho\varphi)-\kappa\varrho^{2}\Delta\varrho\textrm{div}\varphi-2\kappa\varrho\left(\varphi\cdot\nabla\varrho\right)\Delta\varrho\right]dxdt=0,
\end{equation}
where, for $(i,l)=1,2,3$, the viscous term reads as follows
$$
-2\nu \int_{0}^{T}\int_{\Omega} \varrho D\left(u\right) \cdot \nabla(\varrho\varphi) dx dt=
$$
$$
-\nu\int_{0}^{T}\int_{\Omega}\sqrt{\varrho}\sqrt{\varrho}u_i
(\partial_{li}\varrho \varphi_l
+\partial_{i}\varrho \partial_{l}\varphi_l
+\partial_{l}\varrho \partial_{i}\varphi_l
+\varrho \partial_{li}\varphi_l)dxdt
$$
$$
-\nu\int_{0}^{T}\int_{\Omega}\sqrt{\varrho}\sqrt{\varrho}u_l(\partial_{ii}\varrho \varphi_l + 2\partial_{i}\varrho \partial_{i} \varphi_l + \varrho \partial_{ii} \varphi_l)dxdt
$$
$$
-2\nu\int_{0}^{T}\int_{\Omega}\sqrt{\varrho}u_l \partial_{i}\sqrt{\varrho}
(\partial_{i}\varrho \varphi_l
+\varrho \partial_{i}\varphi_l)dxdt
$$
\begin{equation} \label{viscous_t_BDL}
-2\nu\int_{0}^{T}\int_{\Omega}\sqrt{\varrho}u_i \partial_{l}\sqrt{\varrho} 
(\partial_{i}\varrho \varphi_l
+\varrho \partial_{i}\varphi_l)dxdt;
\end{equation}
\end{thm}

The authors in \cite{BDL} used a particular notion of weak solutions that has the advantage to avoid some mathematical difficulties which arise in the definition of the velocity field in the vacuum region.
Indeed, even though weak solutions for the mass and momentum equations can be written in the sense of distribution, Bresch et al. \cite{BDL} (see also J\" ungel \cite{J}) are unable to prove compactness of solutions for vanishing densities. Moreover, as also remarked by J\" ungel \cite{J}, the regularity $\sqrt{\varrho}u\in L^{\infty}\left(0,T;L^{2}\right)$, does not imply compactness for (an approximation of) the convective term $\sqrt{\varrho}u\otimes\sqrt{\varrho}u$. This is the reason why they consider test functions of the form $\varrho\varphi$ in the momentum equation, which are supported on sets of positive $\varrho$. With this choice, the regularity of $\varrho$, which could be proven to belong to $\varrho\in L^{2}\left(0,T;H^{2}\right)$, allows to recover the usual momentum equation on sets of $\varrho>0$. Note that weak formulation with test functions depending on the solutions itself was already introduced in Desjardins and Esteban \cite{DE} for a fluid-structure interaction problem and is used in the context of quantum fluids by J\"ungel \cite{J}. 

In this paper for the definition of weak solution for the pressureless system (see Definition \ref{ws}) we  follow Bresch et al. \cite{BDL} in terms of test functions of the type $\varrho \varphi$, but we   differ from the original definition in \cite{BDL} for the following reasons. The "degenerate" viscosity prevents the velocity field to be uniquely determined in the vacuum region. Indeed, the system (\ref{cont._eps_cap})-(\ref{mom._eps_cap}) lacks bounds for $u$ (neither the velocity field, nor its gradient are defined a.e. in $\Omega$). Consequently, as suggested by Antonelli and Spirito \cite{AS_3}, \cite{AS_2},   the problem is best analyzed in terms of the variables $\sqrt{\varrho}$ and $\Lambda = \sqrt{\varrho} u$. Similarly, for the momentum we have $\varrho u = \sqrt{\varrho} \Lambda$ (however, for the sake of consistency with the literature concerning the Navier-Stokes-Korteweg system, the weak solutions is defined in terms of $\sqrt{\varrho}$ and $\sqrt{\varrho u}$). 
For the above reasons, the viscous stress tensor in the energy inequalities (\ref{ei_weak}) and \eqref{ee} is thought as 
$$
\varrho D(u) = \sqrt{\varrho} \mathcal{S},
$$
where $\sqrt{\varrho}\mathcal{S}=\mbox{Symm}(\nabla(\varrho u)-2\nabla\sqrt{\varrho}\otimes\sqrt{\varrho} u)$. Indeed, it is not clear if weak solutions satisfy the energy inequality in the usual sense, namely
$$
\frac{d}{dt}\int_{\Omega}\frac{1}{2}\left(\varrho \left|u \right|^{2}
+\kappa\left|\nabla\varrho\right|^{2}\right)
+2\nu\int_{\Omega}\varrho\left|D(u)\right|^2dxdt \leq 0.
$$
Antonelli and Spirito \cite{AS_3}, \cite{AS_2} considered the Navier-Stokes-Korteweg system for a viscous compressible fluid with capillarity effects in three space dimensions. They prove compactness and existence of finite energy weak solutions for large initial data, where vacuum regions are allowed in the definition of weak solutions and no additional damping terms are considered. 
We would like to stress that the result of Antonelli and Spirito \cite{AS_3} was announced during the developing of the present work. For this reason and for the time being, we have chosen to develop the analysis using the notion of weak solutions with test functions of the form $\varrho\varphi$. As mentioned in the end of this paper, it will be a matter of future research the extension of this result in the framework of standard weak solutions.

Our purpose is to study the convergence, as $\varepsilon\rightarrow0$,  of the weak solutions of  the system (\ref{cont._1_cap}) - (\ref{mom._1_cap})  towards the pressureless model and to investigate its main properties. Therefore we start  our analysis by proving the  existence of weak solutions of the pressureless system in the case of general initial data. We want to point out that this result is of interest by himself. Indeed, as shown in Theorem \ref{heat_eq} if we start with a density initial datum $\rho_{0}>0$ far from vacuum and an initial velocity function of the form $u_{0}=\nabla\phi(\rho_{0})$, then at any time $t>0$ the density stays far from vacuum. In other words, when the effects of the pressure are not relevant, then the property of no vacuum is preserved at any time $t>0$.  Moreover, in the case of properly chosen initial data, see Theorem \ref{heat_eq_2} we provide the existence of a smooth density function satisfying the heat equation.

\subsection{Weak solutions and main results} \label{results_weak}
For $\varepsilon=0$ the system (\ref{cont._1_cap}) - (\ref{mom._1_cap}) reads
\begin{equation} \label{cont._eps_cap}
\partial_{t}\varrho+\textrm{div}\left(\varrho u\right)=0,
\end{equation}
\begin{equation} \label{mom._eps_cap}
\partial_{t}\left(\varrho u\right)+\textrm{div}\left(\varrho u\otimes u\right)-2\nu\textrm{div}\left(\varrho D\left(u\right)\right)=\kappa\varrho\nabla\Delta\varrho,
\end{equation}
with 
$$\left.\left(\varrho,u\right)\right|_{t=0}=\left(\varrho_{0},u_{0}\right).$$ 

Now, we define the notion of weak solutions for the system (\ref{cont._eps_cap})-(\ref{mom._eps_cap}) we are dealing with
and we introduce the main results.

\begin{defn} \label{ws}
We say that $(\varrho,u)$ is a weak solution of (\ref{cont._eps_cap})-(\ref{mom._eps_cap}) on $(0,T)$ if and only if
\begin{equation} \label{integrab}
\left\{ \begin{array}{c}
\varrho\in L^{2}\left(0,T;H^{2}\left(\Omega\right)\right)\cap L^{\infty}\left(0,T;H^{1}\left(\Omega\right)\right),\\
\nabla\sqrt{\varrho}, \ \sqrt{\varrho}u\in L^{\infty}\left(0,T;L^{2}\left(\Omega\right)\right),\\
\end{array}\right.
\end{equation}
the continuity equation satisfies the following weak formulation for all $\varphi\in C_{c}^{\infty}\left(\left[0,T\right];C^{\infty}\left(\Omega\right)\right)
$ such that $\varphi(T,\cdot)=0$,
\begin{equation} \label{cont_D}
\int_{\Omega}\varrho_{0}\cdot\varphi\left(0,\cdot\right)dx+\int_{0}^{T}\int_{\Omega}\varrho\cdot\partial_{t}\varphi+\int_{0}^{T}\int_{\Omega}\left(\sqrt{\varrho}\sqrt{\varrho}u\right):\nabla\varphi dxdt=0;
\end{equation}
the momentum equations satisfies the following weak formulation for all $\varphi\in C_{c}^{\infty}\left(\left[0,T\right];C^{\infty}\left(\Omega\right)\right)
$ such that $\varphi(T,\cdot)=0$,
$$\int_{\Omega}\varrho_{0}u_{0}\cdot\varrho_{0}\varphi\left(0,\cdot\right)dx+\int_{0}^{T}\int_{\Omega}\left[\varrho^{2}u\cdot\partial_{t}\varphi+\frac{3}{2}\varrho u\otimes\varrho u:\nabla\varphi\right. +\sqrt{\varrho} u \otimes \sqrt{\varrho} u : \varphi \nabla \varrho
$$
\begin{equation} \label{wf_mom.}
\left.-2\nu \varrho D\left(u\right) \cdot \nabla(\varrho\varphi)-\kappa\varrho^{2}\Delta\varrho\textrm{div}\varphi-2\kappa\varrho\left(\varphi\cdot\nabla\varrho\right)\Delta\varrho\right]dxdt=0
\end{equation}
where, for $(i,l)=1,2,3$, the viscous term reads as follows
$$
-2\nu \int_{0}^{T}\int_{\Omega} \varrho D\left(u\right) \cdot \nabla(\varrho\varphi) dx dt=
$$
$$
-\nu\int_{0}^{T}\int_{\Omega}\sqrt{\varrho}\sqrt{\varrho}u_i
(\partial_{li}\varrho \varphi_l
+\partial_{i}\varrho \partial_{l}\varphi_l
+\partial_{l}\varrho \partial_{i}\varphi_l
+\varrho \partial_{li}\varphi_l)dxdt
$$
$$
-\nu\int_{0}^{T}\int_{\Omega}\sqrt{\varrho}\sqrt{\varrho}u_l(\partial_{ii}\varrho \varphi_l + 2\partial_{i}\varrho \partial_{i} \varphi_l + \varrho \partial_{ii} \varphi_l)dxdt
$$
$$
-2\nu\int_{0}^{T}\int_{\Omega}\sqrt{\varrho}u_l \partial_{i}\sqrt{\varrho}
(\partial_{i}\varrho \varphi_l
+\varrho \partial_{i}\varphi_l)dxdt
$$
\begin{equation} \label{viscous_t}
-2\nu\int_{0}^{T}\int_{\Omega}\sqrt{\varrho}u_i \partial_{l}\sqrt{\varrho} 
(\partial_{i}\varrho \varphi_l
+\varrho \partial_{i}\varphi_l)dxdt.
\end{equation}
Moreover there exists $\Lambda$ such that $\varrho u=\sqrt{\varrho}\Lambda$, and $\mathcal{S}\in L^2((0,T)\times\Omega)$ such that $\sqrt{\varrho}\mathcal{S}=\mbox{Symm}(\nabla(\varrho u)-2\nabla\sqrt{\varrho}\otimes\sqrt{\varrho} u)$ in $\mathcal{D}'$, satisfying the following energy inequality
$$
\sup_{t\in(0,T)}\int_{\Omega}\frac{1}{2}\left(\left|\Lambda\right|^{2}
+\kappa\left|\nabla\varrho\right|^{2}\right)
+2\nu\int_{0}^T\int_{\Omega}\left|\mathcal{S}\right|^2dxdt
$$
\begin{equation} \label{ee}
\leq \int_{\Omega}\frac{1}{2}\left(\varrho_{0}\left|u_{0}\right|^{2}+\kappa\left|\nabla\varrho_{0}\right|^{2}\right)dx;
\end{equation}
and there exists $\mathcal{A}\in L^2((0,T)\times\Omega)$ such that $\sqrt{\varrho}\mathcal{A}=\mbox{Asymm}(\nabla(\varrho u)-2\nabla\sqrt{\varrho}\otimes\sqrt{\varrho} u)$ in $\mathcal{D}'$, such that the following Bresch-Desjardins entropy inequality is satisfied
$$
\sup_{t\in(0,T)}\int_{\Omega}\frac{1}{2}\left(\left|\Lambda
+\nu\nabla\sqrt{\varrho}\right|^{2}
+\kappa\left|\nabla\varrho\right|^{2}\right)
$$
$$
+2\nu\int_{0}^T\int_{\Omega}\left|\mathcal{A}\right|^2dxdt
+\nu\kappa\int_{0}^T\int_{\Omega}\left|\nabla\nabla\varrho\right|^2dxdt
$$
\begin{equation} \label{BD-entropy}
\leq \int_{\Omega}\frac{1}{2}\left(\left|\sqrt{\varrho_{0}}u_{0}+
\nu\nabla\sqrt{\varrho_{0}}\right|^{2}+\kappa\left|\nabla\varrho_{0}\right|^{2}\right)dx;
\end{equation}
\end{defn}  

\bigskip
Now, we introduce our main results. The first Theorem concerns the existence of global weak solutions for the system (\ref{cont._eps_cap}) - (\ref{mom._eps_cap}). 
\begin{thm} \label{th}
	Let $\Omega=\mathbb{T}^{d}$ ($d=2$ or 3) be a periodic domain. Assume that the initial energy
	\begin{equation} \label{ei}
		E_{0}=\int_{\Omega}\frac{1}{2}\left(\varrho_{0}\left|u_{0}\right|^{2}+\kappa\left|\nabla\varrho_{0}\right|^{2}\right)dx
	\end{equation}
	and the quantity
	\begin{equation} \label{f}
		F_{0}=2\nu^{2}\int_{\Omega}\left|\nabla\sqrt{\varrho_{0}}\right|^{2}=\frac{1}{2}\int_{\Omega}\varrho_{0}\left|\nu\nabla\log\varrho_{0}\right|^{2}
	\end{equation}
	are finite. Then, there exists a global weak solution $(\varrho,u)$ of the system (\ref{cont._eps_cap}) - (\ref{mom._eps_cap}) in the sense of Definition \ref{ws}.
\end{thm}

The second results concerns the link between the system (\ref{cont._eps_cap}) - (\ref{mom._eps_cap}) and the heat equation when we choose an initial density $\rho_{0}>0$ and the  
the velocity is expressed as a gradient of a given potential $\phi=\phi\left(\varrho\right)$, function of the density, and satisfying the relation 
\begin{equation} \label{phi_prime}
\phi^{\prime}\left(\varrho\right)=\frac{2\mu^{\prime}\left(\varrho\right)}{\varrho}.
\end{equation}
\begin{thm} \label{heat_eq}
	Let $\Omega=\mathbb{T}^{d}$ ($d=2$ or 3) be a periodic domain. Let	$\varrho_{0}\in L^{1}\left(\Omega\right)$ with $\varrho_{0}>0$ and continuous. Assume also that $u_{0}=-\nabla\phi\left(\varrho_{0}\right)$. Then,  there exists a global weak solution $\left(\varrho,u=-\nabla\phi\left(\varrho\right)\right)
	$ of the system (\ref{cont._eps_cap}) - (\ref{mom._eps_cap}), with $\left(\varrho,u\right)\in C^{\infty}\left(0,T;\Omega\right)
	$ and $\varrho$ solving the following heat equation almost everywhere	
	\begin{equation} \label{heat}
	\partial_{t}\varrho-2\nu\Delta\varrho=0,
	\end{equation}
	$$\varrho\left(0,\cdot\right)=\varrho_{0}.
	$$
\end{thm}
 As already mentioned   as a byproduct  of the previous theorem we  have the property that for pressureless Korteweg systems vacuum zones are not allowed if we start away from them. The link between the system (\ref{cont._eps_cap}) - (\ref{mom._eps_cap}) and the heat equation can be extended  to  the framework of weak solutions when $\varrho_{0}$ is not assumed to be strictly positive, this will be done in the next theorem.
 
 \begin{thm} \label{heat_eq_2}
	Let $\Omega=\mathbb{T}^{d}$ ($d=2$ or 3) be a periodic domain. Let	$\varrho_{0}\in L^{1}\left(\Omega\right)$ with $\varrho_{0}\geq0$ and continuous. Assume moreover that $\left(\varrho_{0},u_{0}\right)$ verify the initial conditions (\ref{ei}) and (\ref{f}) with $u_{0}=-\nabla\phi\left(\varrho_{0}\right)$. Then, there exists a global weak solution $\left(\varrho,u\right)$ of the system (\ref{cont._eps_cap}) - (\ref{mom._eps_cap}) with $\varrho$ solving (\ref{heat}).  
\end{thm}

Finally, the last result concerns the convergence of the weak solution of (\ref{cont._1_cap}) - (\ref{mom._1_cap}) to a weak solution of (\ref{cont._eps_cap}) - (\ref{mom._eps_cap}) in the limit as $\varepsilon\rightarrow0$.
\begin{thm} \label{th.1}
	Let $\Omega=\mathbb{T}^{d}$ ($d=2$ or 3) be a periodic domain. Assume that  the condition on the initial data (\ref{in_en}) and (\ref{f_BDL}) are satisfied. Then, as $\varepsilon\rightarrow0$, a global weak solution $(\varrho_\varepsilon,u_\varepsilon)$ of the system (\ref{cont._1_cap}) - (\ref{mom._1_cap}) converges (in a distribution sense) to a weak solution $(\varrho,u)$ of (\ref{cont._eps_cap}) - (\ref{mom._eps_cap}).
\end{thm}

\subsection{Existence results for the pressureless system}
The first result deals with the global existence of the weak solution for the system (\ref{cont._eps_cap}) - (\ref{mom._eps_cap}).

The proof will follow the framework proposed by Bresch et al. \cite{BDL}. We will show how the gain in the regularity of the density given by the presence of the capillary term will give us more information on the density itself in order to have enough compactness to pass to the limit in the approximation scheme. 

We will assume that a sequence $\left(\varrho_{n},u_{n}\right)_{n\in N}$ of approximate weak solutions, which satisfy the energy inequality \eqref{ee} and  have enough regularity to justify the estimates of this section, have been constructed. 

\subsection{Preliminary lemmas}
In this section we recall some preliminary lemmas that will be useful in the proof of the Theorem \ref{th}.
\begin{lemma} \label{lm1.}
	Let $\left(\varrho_{n},u_{n}\right)$ be a smooth solution of (\ref{cont._eps_cap}) - (\ref{mom._eps_cap}). Then, the following identity holds
	$$
	\frac{1}{2}\frac{d}{dt}\int_{\Omega}\varrho_n\left|\nabla\log\varrho_n\right|^{2}dx+\int_{\Omega}\nabla\textrm{div}u_n\cdot\nabla\varrho_n dx
	$$
	\begin{equation} \label{rho_id.}
	+\int_{\Omega}\varrho_n D\left(u_n\right):\nabla\log\varrho_n\otimes\nabla\log\varrho_n dx=0.
	\end{equation}
\end{lemma}
\begin{proof}
	See Bresch et al. \cite{BDL}.
\end{proof}

\begin{lemma} \label{lm2.}
	Let $\left(\varrho_{n},u_{n}\right)$ be a smooth solution of (\ref{cont._eps_cap}) - (\ref{mom._eps_cap}). Then, the following identity holds
	$$
	\frac{d}{dt}\int_{\Omega}\nu^{2}\varrho_n\left|\nabla\log\varrho_n\right|^{2}dx+\nu\kappa\int_{\Omega}\left|\nabla\nabla\varrho_n\right|^{2}dx
	$$
	\begin{equation} \label{mom_id.}
	=-\frac{d}{dt}\int_{\Omega}\nu u_n \cdot\nabla\varrho_n dx+\int_{\Omega}\nu\varrho_n\nabla u_n:^{t}\nabla u_n dx.
	\end{equation}
\end{lemma}
\begin{proof}
	We multiply the momentum equation by $\nu\nabla\varrho_n/\varrho_n$ and we integrate in space, to get
	$$\int_{\Omega}\nu\varrho_n\left(\partial_{t}u_n+u_n\cdot\nabla u_n\right)\cdot\frac{\nabla\varrho_n}{\varrho_n}dx+\int_{\Omega}2\nu^{2}D\left(u_n\right):\left(\nabla\nabla\varrho_n-\frac{\nabla\varrho_n\otimes\nabla\varrho_n}{\varrho_n}\right)dx
	$$
	$$
	+\int_{\Omega}\nu\kappa\left|\nabla\nabla\varrho_n\right|^{2}dx=0,
	$$
	where we used the continuity equation.
	By using (\ref{rho_id.}) multiplied by $2\nu^2$ and the relation above, we have
	\begin{equation} \label{step_1}
	\frac{d}{dt}\int_{\Omega}\nu^{2}\varrho_n\left|\nabla\log\varrho_n\right|^{2}dx+\int_{\Omega}\nu\kappa\left|\nabla\nabla\varrho_n\right|^{2}dx=I,
	\end{equation}
	where $I$ is given by
	$$I=-\int_{\Omega}\nu\partial_{t}u_n\cdot\nabla\varrho_n-\int_{\Omega}2\nu^{2}\nabla div u_n\cdot\nabla\varrho_n
	$$
	$$
	-\int_{\Omega}\nu\left(u_n \cdot\nabla u_n \right)\cdot\nabla\varrho_n-\int_{\Omega}2\nu^{2}D(u_n):\nabla\nabla\varrho_n
	$$
	$$=-\nu\frac{d}{dt}\int_{\Omega}u_n\cdot\nabla\varrho_n dx+\nu\int_{\Omega}u_n\cdot\nabla\partial_{t}\varrho_n-2\nu^{2}\int_{\Omega}\nabla div u_n \cdot\nabla\varrho_n
	$$
	$$
	-\int_{\Omega}\nu\left(u_n\cdot\nabla u_n\right)\cdot\nabla\varrho_n-\int_{\Omega}2\nu^{2}D(u_n):\nabla\nabla\varrho_n
	$$
	$$=-\nu\frac{d}{dt}\int_{\Omega}u_n\cdot\nabla\varrho_n dx-\nu\int_{\Omega}u_n\cdot\nabla div (\varrho_n u_n)-2\nu^{2}\int_{\Omega}\nabla div u_n\cdot\nabla\varrho_n
	$$
	\begin{equation} \label{I}
	-\int_{\Omega}\nu\left(u_n\cdot\nabla u_n\right)\cdot\nabla\varrho_n-\int_{\Omega}2\nu^{2}D(u_n):\nabla\nabla\varrho_n,
	\end{equation}
	and  we used the mass equation in the last equality.
	Integration by parts and the identity $\nabla div=curl curl-\Delta
	$ gives respectively
	\begin{equation} \label{step_2}
	-\int_{\Omega}u_n \cdot\nabla div (\varrho_n u_n)dx-\int_{\Omega}\left(u_n\cdot\nabla u_n\right)\cdot\nabla\varrho_n dx=\int_{\Omega}\varrho_n\nabla u_n:^{t}\nabla u_n dx
	\end{equation}
	and
	\begin{equation} \label{step_3}
	-\int_{\Omega}\nabla div u_n\cdot\nabla\varrho_n dx-\int_{\Omega}D(u_n):\nabla\nabla\varrho_n dx=0.
	\end{equation}
	Relations (\ref{step_1})-(\ref{step_3}) give (\ref{mom_id.}).
\end{proof}

The following Lemma concerns the Bresch-Desjardins entropy relation for the pressureless system (\ref{cont._eps_cap}) - (\ref{mom._eps_cap}).

\begin{lemma}
	Let $\left(\varrho_{n},u_{n}\right)$ be a smooth solution of (\ref{cont._eps_cap}) - (\ref{mom._eps_cap}). Then, defining $w_n=u_n + \nu \nabla \varrho_n$, the following inequality holds
$$
\sup_{t\in(0,T)}\int_{\Omega}\frac{1}{2}\left(\left|\varrho_n| w_n\right|^{2}
+\kappa\left|\nabla\varrho_n\right|^{2}\right)
$$
$$
+2\nu\int_{0}^T\int_{\Omega}\varrho_n\left|Au_n\right|^2dxdt
+\int_{0}^T\int_{\Omega}\left|\nabla\nabla\varrho_n\right|^2dxdt
$$
\begin{equation} \label{BD-entropy}
\leq \int_{\Omega}\frac{1}{2}\left(\varrho_{0,n}\left|w_{0,n}\right|^{2}+\kappa\left|\nabla\varrho_{0,n}\right|^{2}\right)dx
\end{equation}
with $Au_n=(\nabla u_n - ^t\nabla u_n)/2$.
\end{lemma}

\begin{proof}
By simple algebra and elementary identities (see Antonelli and Spirito \cite{AS_2}, Preposition 3.2; see also Bresch and Desjardins \cite{BDL_0}), is easy to derive the following system
\begin{equation} \label{cont_wn}
    \partial_{t}\varrho_n + div (\varrho_n w_n)= 2\nu \Delta \varrho_n,
\end{equation}
\begin{equation} \label{mom_wn}
    \partial_{t} (\varrho_n w_n) + div (\varrho_n w_n \otimes w_n) - 2\nu \Delta (\varrho_n w_n) + 2\nu div (\varrho_n D w_n) = 2\kappa \varrho_n \nabla \Delta \varrho_n. 
\end{equation}
We multiply (\ref{mom_wn}) by $w_n$ and we integrate in space. Using (\ref{cont_wn}), we have
\begin{equation} \label{mult_1}
    \frac{d}{dt}\int_{\Omega} \varrho_n \frac{\left| w_n \right|^2}{2} dx
    +2\nu \int_{\Omega} \varrho_n \left| Au_n \right|^2 dx - 2\kappa \int_{\Omega} \varrho_n \nabla \Delta \varrho_n w_n dx =0.
\end{equation}
Now, we multiply (\ref{cont_wn}) by $-2\kappa \Delta \varrho_n$. We have,
\begin{equation} \label{mult_2}
    \frac{d}{dt} \int_{\Omega} \kappa \left|  \nabla \varrho_n \right|^2 + 4\nu\kappa \int_{\Omega} \left|  \Delta \varrho_n \right|^2 dx -2\kappa \int_{\Omega} div (\varrho_n w_n) \Delta \varrho_n dx.
\end{equation}
By summing up (\ref{mult_1}) with (\ref{mult_2}) and integrating by parts we obtain (\ref{BD-entropy}).
\end{proof}

\subsubsection{Proof of Theorem \ref{th}}

The proof of Theorem \ref{th} is divided in several steps. First, we collect a priori estimates, then we will deal with the strong convergence of the density, the analysis of the momentum and finally with the convergence of the remaining non-linear terms.

\bigskip
\textit{- Step 1. A priori estimates}

\bigskip
From the continuity equation (\ref{cont._eps_cap}) and the energy equality (\ref{ee}) we can deduce the following a priori estimates
$$
\left\Vert \varrho_{n}\right\Vert _{L^{\infty}\left(0,T;L^{1}\left(\Omega\right)\right)}\leq C, \ \  \left\Vert \sqrt{\varrho_{n}}u_{n}\right\Vert _{L^{\infty}\left(0,T;L^{2}\left(\Omega\right)\right)}\leq C,
$$
\begin{equation} \label{est_1}
\left\Vert \nabla\varrho_{n}\right\Vert _{L^{\infty}\left(0,T;L^{2}\left(\Omega\right)\right)}\leq C, \ \ \left\Vert \sqrt{\varrho_{n}}D\left(u_{n}\right)\right\Vert _{L^{2}\left(0,T;L^{2}\left(\Omega\right)\right)}\leq C.
\end{equation}
Moreover, Lemma \ref{lm2.} yields
\begin{equation} \label{est_2}
\left\Vert \varrho_{n}\right\Vert _{L^{2}\left(0,T;H^{2}\left(\Omega\right)\right)}\leq C, \ \ \left\Vert \nabla\sqrt{\varrho_{n}}\right\Vert _{L^{\infty}\left(0,T;L^{2}\left(\Omega\right)\right)}\leq C. 
\end{equation}
Finally, by the second estimate in (\ref{est_2}) and the first estimate in (\ref{est_1}), we can conclude that
\begin{equation} \label{est_4}
\left\Vert \sqrt{\varrho_{n}}\right\Vert _{L^{\infty}\left(0,T;W^{1,2}\left(\Omega\right)\right)}\leq C.
\end{equation}	

\bigskip
\textit{- Step 2. Convergence of the density}

\bigskip
Now, thanks to (\ref{est_1}) and (\ref{est_4}), we deduce that $\varrho_{n}u_{n}=\sqrt{\varrho_{n}}\sqrt{\varrho_{n}}u_{n}$ is bounded in $L^{\infty}\left(0,T;L^{3/2}\left(\Omega\right)\right)$.  
The continuity equation thus gives $\partial_{t}\varrho_{n}$ bounded in $L^{\infty}\left(0,T;W^{-1,2}\left(\Omega\right)\right)$.  
Then, the Aubin-Lions Lemma gives the following strong convergence of the density, $$\varrho_{n}\rightarrow\varrho, \ \textrm{in}
\ L^{2}\left(0,T;L^{2}\left(\Omega\right)\right).$$
Indeed, we have more regularity on the density, and the following convergence holds (see Bresch et al. \cite{BDL}),
\begin{equation} \label{reg_dens}
\varrho_{n}\rightarrow\varrho \ \ \textrm{in} \ \ L^{2/s}\left(0,T;H^{1+s}\left(\Omega\right)\right)\cap C\left(\left[0,T\right];H^{s}\right) \ \ \textrm{for} \ \ \textrm{all}\ \ s\in\left(0,1\right).
\end{equation} 
Moreover, because of the strong convergence of $\varrho_{n}^{2}$ and $\varrho_{n}\nabla\varrho_{n}$ in $L^{2}\left(0,T;L^{2}\left(\Omega\right)\right)$ and the weak convergence of $\Delta\varrho_{n}$ in $L^{2}\left(0,T;L^{2}\left(\Omega\right)\right)$, we are allowed to pass to the limit in the last two terms of (\ref{wf_mom.}). 

\bigskip
\textit{- Step 3. Limit velocity}

\bigskip
Now, since from (\ref{reg_dens}) $\varrho_{n}$ converges almost everywhere in $\left(0,T\right)\times\Omega$ and from the second bound in (\ref{est_1}),  $\sqrt{\varrho_{n}}u_{n}$ converges weakly to some $g$ in $L^{2}\left(0,T;L^{2}\left(\Omega\right)\right)$, in the spirit of Bresch et al. \cite{BDL} this allows to define a limit velocity $u=g/\sqrt{\varrho}$ on the set of positive $\varrho$, and zero otherwise. 
Then, since $\varrho_{n}u_{n}=\sqrt{\varrho_{n}}\sqrt{\varrho_{n}}u_{n}$ converges weakly to $\sqrt{\varrho}g=\varrho u$, we proved $$\partial_{t}\varrho+\textrm{div}\left(\sqrt{\varrho}\sqrt{\varrho} u\right)=0, \ \ \left.\varrho\right|_{t=0}=\varrho_{0} \ \ \textrm{in} \ \ \mathcal{D}^{\prime}\left(\Omega\right).$$ Moreover, we are able to pass to the limit in the first two terms of (\ref{wf_mom.}). For the first one, we assume the strong convergence of the initial data, and for the second one  we use the strong convergence of $\varrho_{n}$ in $C\left(0,T;L^{3}\left(\Omega\right)\right)$ from (\ref{reg_dens}) combined with the weak convergence of $\varrho_{n}u_{n}$ in $L^{\infty}\left(0,T;L^{3/2}\left(\Omega\right)\right).$ 

Note that, although we can define $u=g/\sqrt{\varrho}$ outside the vacuum set, we do not know if $\sqrt{\varrho}u$ is zero on the vacuum set, hence  it is not clear weather
$$
\sqrt{\varrho_n} u_n \rightarrow \sqrt{\varrho} u \ \mbox{weakly in} \ L^{2}((0,T)\times \Omega).
$$

\bigskip
\textit{- Step 4. Convergence of the momentum}

\bigskip
In order to prove the strong convergence of the momentum $\varrho_{n}u_{n}$ to $\varrho u$ in $L^{2}\left(0,T;L^{2}\left(\Omega\right)\right)$, we observe that
$$\varrho_{n}u_{n}\otimes\varrho_{n}u_{n}=\varrho_{n}^{3/2}u_{n}\otimes\varrho_{n}^{1/2}u_{n}.
$$
Because of the weak convergence of $\sqrt{\varrho_{n}}u_{n}$ to $\sqrt{\varrho}u$ in $L^{2}\left(0,T;L^{2}\left(\Omega\right)\right)$, it is enough to prove the strong convergence of $\varrho_{n}^{3/2}u_{n}$ to $\varrho^{3/2}u$ in $L^{2}\left(0,T;L^{2}\left(\Omega\right)\right)$.
First, we notice that
$$D\left(\varrho_{n}^{3/2}u_{n}\right)=\varrho_{n}\sqrt{\varrho_{n}}D(u_{n})+\frac{3}{2}\sqrt{\varrho_{n}}u_{n}\underline{\otimes}\nabla\varrho_{n},
$$ where $a\underline{\otimes}b=\left(a\otimes b+b\otimes a\right)/2$. From the uniform bound of $\varrho_{n}$ and $\sqrt{\varrho_{n}}D(u_{n})$ respectively in $L^{\infty}\left(0,T;L^{6}\left(\Omega\right)\right)$ and $L^{2}\left(0,T;L^{2}\left(\Omega\right)\right)$, and of $\sqrt{\varrho_{n}}u_{n}$ and $\nabla\varrho_{n}$ in $L^{\infty}\left(0,T;L^{2}\left(\Omega\right)\right)$ and $L^{2}\left(0,T;L^{6}\left(\Omega\right)\right)$, we deduce that $D\left(\varrho_{n}^{3/2}u_{n}\right)$ is bounded uniformly in $L^{2}\left(0,T;L^{3/2}\left(\Omega\right)\right)$. Observing that the following holds,
$$\left\Vert \nabla\left(\varrho_{n}^{3/2}u_{n}\right)\right\Vert _{L^{3/2}\left(\Omega\right)}\leq\left\Vert D\left(\varrho_{n}^{3/2}u_{n}\right)\right\Vert _{L^{3/2}\left(\Omega\right)},
$$
by  the Sobolev embeddings $\varrho_{n}^{3/2}u_{n}$ is uniformly bounded in $L^{2}\left(0,T;L^{3}\left(\Omega\right)\right)$. 
Now, let us write
$$\left\Vert \varrho_{n}^{3/2}u_{n}-\varrho^{3/2}u\right\Vert _{L^{2}\left(0,T;L^{2}\left(\Omega\right)\right)}\leq\left\Vert \varrho_{n}^{3/2}u_{n}-\left(\varrho_{n}^{3/2}u_{n}\right)\ast\psi_{m}\right\Vert _{L^{2}\left(0,T;L^{2}\left(\Omega\right)\right)}
$$
$$+\left\Vert \left(\varrho_{n}^{3/2}u_{n}\right)\ast\psi_{m}-\left(\varrho^{3/2}u\right)\ast\psi_{m}\right\Vert _{L^{2}\left(0,T;L^{2}\left(\Omega\right)\right)}
$$
$$+\left\Vert \varrho^{3/2}u-\left(\varrho^{3/2}u\right)\ast\psi_{m}\right\Vert _{L^{2}\left(0,T;L^{2}\left(\Omega\right)\right)}
$$
where $\psi\in C^{\infty}\left(\Omega\right)$ is a mollifying kernel such that $\psi\geq0$, $\int_{\Omega}\psi dx=1$ and $\psi_{m}\left(\cdot\right)=m^{d}\psi\left(m\cdot\right)$ for all $m\in\mathbb{N}
$.
We have,
\begin{equation} \label{phi_m}
\left\Vert \varrho_{n}^{3/2}u_{n}-\left(\varrho_{n}^{3/2}u_{n}\right)\ast\psi_{m}\right\Vert _{L^{2}\left(0,T;L^{2}\left(\Omega\right)\right)}\leq\frac{C}{\sqrt{m}}\left\Vert \nabla\left(\varrho_{n}^{3/2}u_{n}\right)\right\Vert _{L^{2}\left(0,T;L^{3/2}\left(\Omega\right)\right)}
\end{equation}
and similarly for $\varrho^{3/2}u$. Next, for $l\in\mathbb{N}$, we have,
$$\left\Vert \left(\varrho_{n}^{3/2}u_{n}\right)\ast\psi_{m}-\left(\varrho^{3/2}u\right)\ast\psi_{m}\right\Vert _{L^{2}\left(0,T;L^{2}\left(\Omega\right)\right)}$$ 
\begin{equation} \label{psi_l}
\leq C_{m}\left(\frac{1}{l}+\left\Vert \left(\varrho_{n}^{3/2}u_{n}\right)\ast\psi_{l}-\left(\varrho^{3/2}u\right)\ast\psi_{l}\right\Vert _{L^{2}\left(0,T;L^{1}\left(\Omega\right)\right)}\right).
\end{equation}
Now, in order to handle and to separate the analysis for densities close to vacuum and bounded away from zero, we introduce a cut-off function $\beta\in C^{\infty}(\mathbb{R})$ such that $\beta(s)=1$ for $s\geq2$ and $\beta(s)=0$ for $s\leq1$, with $0\leq\beta(s)\leq1$, and for any $\alpha>0$ we define $\beta_{\alpha}(s)=\alpha^{-d}\beta\left(s/\alpha\right)$. We have
$$
\left\Vert \left(\varrho_{n}^{3/2}u_{n}\left(1-\beta_{\alpha}\left(\varrho_{n}\right)\right)\right)\ast\psi_{l}\right\Vert _{L^{2}\left(0,T;L^{1}\left(\Omega\right)\right)}
$$
\begin{equation} \label{dens_cut_1}
\leq\left\Vert \sqrt{\varrho_{n}}u_{n}\right\Vert _{L^{\infty}\left(0,T;L^{2}\left(\Omega\right)\right)}\left\Vert \varrho_{n}\left(1-\beta_{\alpha}\left(\varrho_{n}\right)\right)\right\Vert _{L^{2}\left(0,T;L^{2}\left(\Omega\right)\right)}\leq C\alpha
\end{equation}
and
$$\left\Vert \left(\varrho_{n}^{3/2}u_{n}\beta_{\alpha}\left(\varrho_{n}\right)\right)\ast\psi_{l}-\varrho_{n}^{-1/2}\beta_{\alpha}\left(\varrho_{n}\right)\left(\varrho_{n}^{2}u_{n}\right)\ast\psi_{l}\right\Vert _{L^{2}\left(0,T;L^{1}\left(\Omega\right)\right)}
$$
$$
\leq\frac{C}{l}\left\Vert \varrho_{n}^{2}u_{n}\right\Vert _{L^{2}\left(0,T;L^{2}\left(\Omega\right)\right)}\left\Vert \nabla\left(\varrho_{n}^{-1/2}\beta_{\alpha}\left(\varrho_{n}\right)\right)\right\Vert _{L^{\infty}\left(0,T;L^{2}\left(\Omega\right)\right)}
$$
\begin{equation} \label{dens_cut_2}
\leq\frac{C_{\alpha}}{l}\left\Vert \sqrt{\varrho_{n}}\right\Vert _{L^{\infty}\left(0,T;L^{6}\left(\Omega\right)\right)}\left\Vert \varrho_{n}^{3/2}u_{n}\right\Vert _{L^{2}\left(0,T;L^{3}\left(\Omega\right)\right)}\left\Vert \nabla\varrho_{n}\right\Vert _{L^{\infty}\left(0,T;L^{2}\left(\Omega\right)\right)}.
\end{equation}
By considering (\ref{dens_cut_1}) and (\ref{dens_cut_2}), we obtain
$$\left\Vert \varrho_{n}^{3/2}u_{n}-\varrho^{3/2}u\right\Vert _{L^{2}\left(0,T;L^{2}\left(\Omega\right)\right)}
$$
$$
\leq C\left(\frac{1}{\sqrt{m}}+\alpha\right)+\frac{C_{m,\alpha}}{l}
$$
$$+\left\Vert \varrho_{n}^{-1/2}\beta_{\alpha}\left(\varrho_{n}\right)\left(\varrho_{n}^{2}u_{n}\right)\ast\psi_{l}-\varrho^{-1/2}\beta_{\alpha}\left(\varrho\right)\left(\varrho^{2}u\right)\ast\psi_{l}\right\Vert _{L^{2}\left(0,T;L^{1}\left(\Omega\right)\right)}
$$
for a given $\alpha$, $m$ and $l$. Now, thanks to the strong convergence of $\varrho_{n}^{-1/2}\beta_{\alpha}\left(\varrho_{n}\right)$ to $\varrho^{-1/2}\beta_{\alpha}\left(\varrho\right)$ in $C\left(0,T;L^{2}\left(\Omega\right)\right)$, the next step consists in proving the strong convergence of $\left(\varrho_{n}^{2}u_{n}\right)\ast\psi_{l}$ to $\left(\varrho^{2}u\right)\ast\psi_{l}$ in $L^{2}\left(0,T;L^{2}\left(\Omega\right)\right)$. Since $\varrho_{n}^{2}u_{n}$ is bounded in $L^{2}\left(0,T;L^{2}\left(\Omega\right)\right) $, we need only to prove that $\partial_{t}\left(\varrho_{n}^{2}u_{n}\right)$ is bounded in $L^{q}\left(0,T;H^{-s}\left(\Omega\right)\right)$ for $q>1$ and some $s>0$. 

To this aim we rewrite the weak formulation \eqref{wf_mom.} in the following equivalent form, 
$$
\int_{0}^{T}\int_{\Omega}\partial_{t}(\varrho^{2}_n u_n)\varphi=-\int_{0}^{T}\int_{\Omega}\left[\varrho_n u_n\otimes\varrho_n u_n:\nabla\varphi\right.
$$
$$
-\varrho^2_n (u_n \cdot \varphi)\textrm{div}u_n 
-2\nu \varrho_n D\left(u_n \right) : \varrho_n  \nabla\varphi
-2\nu \varrho_n D\left(u_n \right) : \varphi \otimes \nabla \varrho_n 
$$
\begin{equation} \label{wf_approx}
\left. -\kappa\varrho^{2}_n\Delta\varrho_n\textrm{div}\varphi-2\kappa\varrho_n\left(\varphi\cdot\nabla\varrho_n\right)\Delta\varrho_n\right]dxdt.
\end{equation}

Now, the $L^{2}\left(0,T;L^{3}\left(\Omega\right)\right)$ bound of $\varrho_{n}^{3/2}u_{n}$ and the $L^{\infty}\left(0,T;L^{2}\left(\Omega\right)\right)$ of $\varrho_{n}^{1/2}u_{n}$, gives $\varrho_{n}u_{n}\otimes\varrho_{n}u_{n}$ bounded in $L^{2}\left(0,T;L^{6/5}\left(\Omega\right)\right)
$.
Next, the quantity $\varrho_{n}^{3/2}u_{n}$ is bounded in $L^{5/2}\left(0,T;L^{5/2}\left(\Omega\right)\right)$ thanks to the interpolation between  
$L^{2}\left(0,T;L^{3}\left(\Omega\right)\right)
$ and $L^{\infty}\left(0,T;L^{3/2}\left(\Omega\right)\right)$. Consequently, we have $\varrho_{n}^{2}u_{n}\textrm{div}u_{n}$ bounded in $L^{10/9}\left(0,T;L^{10/9}\left(\Omega\right)\right)$. Next, the bound of $\varrho_{n}^{3/2}$ in $L^{\infty}\left(0,T;L^{4}\left(\Omega\right)\right)$ and of $\sqrt{\varrho_{n}}\nabla\varrho_{n}$ in $L^{4}\left(0,T;L^{2}\left(\Omega\right)\right)$ gives $\varrho_{n}^{2}D\left(u_{n}\right)$ bounded in $L^{2}\left(0,T;L^{4/3}\left(\Omega\right)\right)
$ and $\varrho_{n}D\left(u_{n}\right)_{ij}\partial_{j}\varrho_{n}$ in $L^{4/3}\left(0,T;L^{1}\left(\Omega\right)\right)$. Finally, $\varrho_{n}^{2}\Delta\varrho_{n}
$ and $\varrho_{n}\Delta\varrho_{n}\nabla\varrho_{n}
$ are bounded, respectively, in $L^{2}\left(0,T;L^{1}\left(\Omega\right)\right)
$ and $L^{4/3}\left(0,T;L^{1}\left(\Omega\right)\right)
$.
Then, it follows from (\ref{wf_approx}) that
\begin{equation} \label{dt_rhou}
\left|\int_{0}^{T}\int_{\Omega}\partial_{t}\left(\varrho_{n}^{2}u_{n}\right)\cdot\varphi dxdt\right|\leq C\left(\left\Vert \varphi\right\Vert _{L^{r}\left(0,T;L^{\infty}\left(\Omega\right)\right)}+\left\Vert \nabla\varphi\right\Vert _{L^{r}\left(0,T;L^{\infty}\left(\Omega\right)\right)}\right),
\end{equation}
with $r<+\infty$. This means that $\partial_{t}\left(\varrho_{n}^{2}u_{n}\right)$ is bounded in $L^{q}\left(0,T;H^{-s}\left(\Omega\right)\right)$ for $q>1$ and $s>5/2$. Consequently, $\varrho_{n}^{3/2}u_{n}$ converges strongly to $\varrho^{3/2}u$ in $L^{2}\left(0,T;L^{2}\left(\Omega\right)\right)$, and $\varrho_{n}u_{n}$ strongly to $\varrho u$ in $L^{2}\left(0,T;L^{2}\left(\Omega\right)\right)$.

\bigskip
\textit{- Step 5. Convergence of the remaining non-linear terms}

\bigskip
The convergence of the  term 
\begin{equation} \label{convergence_nlt}
\int_{0}^T \int_\Omega \sqrt{\varrho_n} u_n \otimes \sqrt{\varrho_n} u_n : \varphi \nabla\varrho_n.
\end{equation}
is obtained thanks to the strong convergence of the quantities $\varrho_n u_n$ and $\nabla \varrho_n$, together with the weak convergence of the velocity field $u_n$.

Concerning the viscous terms on the right-hand-side of (\ref{viscous_t}), the convergence of the first two quantities is allowed by the strong convergence of the density $\varrho_n$, the momentum $\varrho_n u_n$ and the gradient of the density $\nabla \varrho_n$, together with the weak convergence of $\nabla \nabla \varrho_n$ and $\Delta \varrho_n$. Then, the strong convergence of $\sqrt{\varrho_n}$ together with the weak convergence of $\nabla \sqrt{\varrho_n}$ allow to pass to the limit in the remaining two quantities.

\subsection{Pressureless system, vacuum  and the heat equation} \label{heat.}
This section is devoted to the proof of the Theorems \ref{heat_eq} and \ref{heat_eq_2}. Hence for the  special choice of an irrotational  initial velocity, namely expressed as a gradient of a given potential we are able to describe the propagation of the vacuum zones and connection between the solutions of the pressureless system and the heat equation. First, we deal with the case of the initial density $\varrho_{0}\in L^{1}\left(\Omega\right)$ away from the vacuum, namely $\varrho_{0}>0$ and then we extend this result to the general case $\varrho_{0}\geq0$.

\subsubsection{Proof of Theorem \ref{heat_eq}}
	Assume that $\left(\varrho,u\right)$ are classical solutions of (\ref{cont._eps_cap}) - (\ref{mom._eps_cap}). We are looking for a solution of the form $\left(\varrho,-\nabla\phi\left(\varrho\right)\right)$. The continuity equation yields
	\begin{equation} \label{step1}
	\partial_{t}\varrho-\textrm{div}\left(\varrho\nabla\phi\left(\varrho\right)\right)=0.
	\end{equation}
	Since $\phi^{\prime}\left(\varrho\right)=2\nu/\varrho$, we have
	\begin{equation} \label{step2}
	\partial_{t}\varrho-2\nu\Delta\varrho=0.
	\end{equation}
	Now, we have to check that the momentum equation shows a compatibility with the continuity equation keeping an irrotational structure. We have,
	\begin{equation} \label{step3}
	\partial_{t}\left(\varrho u\right)=-\partial_{t}\left(\varrho\nabla\phi\left(\varrho\right)\right)=-2\nabla\partial_{t}\mu\left(\varrho\right).
	\end{equation}
	Indeed, $\varrho\nabla\phi\left(\varrho\right)=\nabla\mu\left(\varrho\right)$.
	Next, we have
	\begin{equation} \label{step4}
	-2\textrm{div}\left(\mu\left(\varrho\right)Du\right)=2\textrm{div}\left(\mu\left(\varrho\right)\nabla\nabla\phi\left(\varrho\right)\right)$$
	$$=2\mu\left(\varrho\right)\nabla\Delta\phi\left(\varrho\right)+2\nabla\mu\left(\varrho\right)\cdot\nabla\nabla\phi\left(\varrho\right),
	\end{equation}
	and
	\begin{equation} \label{step5}
	\textrm{div}\left(\varrho u\otimes u\right)=2\Delta\phi\left(\varrho\right)\nabla\mu\left(\varrho\right)+2\nabla\nabla\mu\left(\varrho\right)\cdot\nabla\phi\left(\varrho\right).
	\end{equation}
	Combining (\ref{step4}) with (\ref{step5}), we get
	\begin{equation} \label{step6}
	\textrm{div}\left(\varrho u\otimes u\right)-2\textrm{div}\left(\mu\left(\varrho\right)Du\right)=2\nabla\left(\mu\left(\varrho\right)\Delta\phi\left(\varrho\right)\right)+2\nabla\left(\nabla\mu\left(\varrho\right)\cdot\nabla\phi\left(\varrho\right)\right).
	\end{equation} 
	Now, using (\ref{step3}), (\ref{step6}) and the momentum equation (\ref{mom._eps_cap}), we have
	\begin{equation} \label{step7}
	\partial_{t}\left(\varrho u\right)+\textrm{div}\left(\varrho u\otimes u\right)-2\textrm{div}\left(\mu\left(\varrho\right)Du\right)$$$$
	=-2\nabla\left(\partial_{t}\mu\left(\varrho\right)-\varrho\mu^{\prime}\left(\varrho\right)\Delta\phi\left(\varrho\right)-\nabla\mu\left(\varrho\right)\cdot\nabla\phi\left(\varrho\right)\right)+\kappa\varrho\nabla\Delta\varrho,
	\end{equation}
	where we used $\mu\left(\varrho\right)=\varrho\mu^{\prime}\left(\varrho\right)$. Now, since $2\mu^{\prime}\left(\varrho\right)=\varrho\phi^{\prime}\left(\varrho\right)$, we can compute
	$$\Delta\mu\left(\varrho\right)=\frac{1}{2}\varrho\Delta\phi\left(\varrho\right)+\frac{1}{2}\nabla\varrho\cdot\nabla\phi\left(\varrho\right),$$
	and, consequently,
	\begin{equation} \label{step8}
	4\mu^{\prime}\left(\varrho\right)\Delta\mu\left(\varrho\right)=2\varrho\mu^{\prime}\left(\varrho\right)\Delta\phi\left(\varrho\right)+2\nabla\mu\left(\varrho\right)\cdot\nabla\phi\left(\varrho\right).
	\end{equation}
	Combining (\ref{step7}) with (\ref{step8}), we obtain
	\begin{equation} \label{step9}
	\partial_{t}\left(\varrho u\right)+\textrm{div}\left(\varrho u\otimes u\right)-2\textrm{div}\left(\mu\left(\varrho\right)Du\right)$$$$
	=-\nabla\left(2\mu^{\prime}\left(\varrho\right)\left(\partial_{t}\varrho-2\Delta\mu\left(\varrho\right)\right)\right)+\kappa\varrho\nabla\Delta\varrho.
	\end{equation}
Since, thanks to the mass equations,  the above equation is compatible with the  momentum equations, the proof is concluded in the smooth case. In the   case of $\rho_{0}>0$, $\rho_{0}\in L^{1}(\Omega)$  the proof follows from  the standard theory of the heat equation (see Evans \cite{E}). Indeed, the parabolic equation immediately smooths any initial data belonging to $L^{p}(\Omega)$, $1\leq p\leq\infty
	$, to a function $\varrho\left(\cdot,t\right)\in C^{\infty}\left(\Omega\right)$. Therefore, the unique solution of  $\partial_{t}\varrho-2\nu\Delta\varrho=0$ for $t>0$ is non-negative and smooth and all the previous formal computations are justified yielding a solution of (\ref{cont._eps_cap}) - (\ref{mom._eps_cap}) with $\varrho>0$ solving (\ref{heat}) and $u=-\nabla\varphi\left(\varrho\right)$. 

\subsubsection{Proof of Theorem \ref{heat_eq_2}}
	
	In order to prove the existence result it suffices to construct a sequence of global regular solutions $(\varrho_{n},u_{n})$ verifying the assumptions on the initial data (\ref{ei}) and (\ref{f}).
	We set $\varrho_{0}^{n}=\varrho_{0}+\frac{\widetilde{\varrho_{0}}}{n}$, where $\widetilde{\varrho_{0}}$ is a continuous, strictly positive functions, such that $\sqrt{\widetilde{\varrho_{0}}}\nabla\phi\left(\widetilde{\varrho_{0}}\right)
	$ and $\left|\nabla\widetilde{\varrho_{0}}\right|$, $\left|\nabla\sqrt{\widetilde{\varrho_{0}}}\right|$ are bounded in $L^{2}\left(\Omega\right)$. From Theorem \ref{heat_eq}, we know that there exists a smooth solutions $(\varrho_{n}, u_{n})$ with initial data $\left(\varrho_{0}^{n},-\nabla\phi\left(\varrho_{0}^{n}\right)\right)$. Moreover we have 
	\begin{equation} \label{exs.}
	|\sqrt{\varrho_{0}^{n}}\nabla\phi\left(\varrho_{0}^{n}\right)|=\left|\sqrt{\varrho_{0}^{n}}\phi^{\prime}\left(\varrho_{0}^{n}\right)\nabla\varrho_{0}^{n}\right|\leq\frac{\left|\nabla\varrho_{0}\right|}{\sqrt{\varrho_{0}}}+\frac{1}{\sqrt{n}}\frac{\left|\nabla\widetilde{\varrho_{0}}\right|}{\sqrt{\widetilde{\varrho_{0}}}}.
	\end{equation}
	It implies that $\sqrt{\varrho_{0}^{n}}\nabla\phi\left(\varrho_{0}^{n}\right)$ is uniformly bounded in $L^{2}\left(\Omega\right)$ by using the fact that $\sqrt{\varrho_{0}}\nabla\phi\left(\varrho_{0}\right)$ and $\sqrt{\widetilde{\varrho_{0}}}\nabla\phi\left(\widetilde{\varrho_{0}}\right)$ are bounded in $L^{2}\left(\Omega\right)$. Next, in a similar way we have also the following $L^{2}$ bounds,
	\begin{equation} \label{exs.1}
	\left\Vert \nabla\varrho_{0}^{n}\right\Vert _{L^{2}\left(\Omega\right)}\leq\left\Vert \nabla\varrho_{0}\right\Vert _{L^{2}\left(\Omega\right)}+\frac{1}{n}\left\Vert \nabla\widetilde{\varrho_{0}}\right\Vert _{L^{2}\left(\Omega\right)},
	\end{equation}
	\begin{equation} \label{exs.1b}
	\left\Vert \nabla\sqrt{\varrho_{0}^{n}}\right\Vert _{L^{2}\left(\Omega\right)}\leq\left\Vert \frac{\nabla\sqrt{\varrho_{0}}}{\sqrt{\varrho_{0}^{n}}}\right\Vert _{L^{2}\left(\Omega\right)}+\frac{1}{n}\left\Vert \frac{\nabla\sqrt{\widetilde{\varrho_{0}}}}{\sqrt{\varrho_{0}^{n}}}
	\right\Vert _{L^{2}\left(\Omega\right)}.
	\end{equation}
	As a consequence $\left(\varrho_{0}^{n},-\nabla\phi\left(\varrho_{0}^{n}\right)\right)$ fulfils (\ref{ei}) and (\ref{f}) and  by Theorem \ref{th}  we obtain  the existence of a global weak solution $(\varrho, u)$ (to be precise, here we exploit the stability result contained in the proof of Theorem \ref{th}).
	
	Now, in order to prove that $(\varrho,u)$ is the unique strong solution of (\ref{heat}) we recall that the heat equation verifies the so-called $L^1$-contraction principle (see for example Chapter 2 of Galaktionov and V\' azquez \cite{GaVa}), which applied to our  sequence 
	$\varrho_{n}$, which has a strictly positive initial data,  reads 
	\begin{equation} \label{contr}
	\left\Vert \varrho_{n}(t,\cdot)-\varrho_{m}(t,\cdot)\right\Vert _{L^{1}\left(\Omega\right)}\leq\left\Vert \varrho_{0}^{n}-\varrho_{0}^{m}\right\Vert _{L^{1}\left(\Omega\right)},
	\end{equation}
	for every $n,m\in\mathbb{N}$. Since we have that  $\varrho_{0}^{n}\rightarrow\varrho_{0}$ in $L^{1}\left(\Omega\right)$, the sequence $\left(\varrho_{n}\right)_{n\in\mathbb{N}}$ is a Cauchy sequence strongly convergent to $\varrho_{1}$ in $L^{1}\left(\Omega\right)$, but since we have proved that $\varrho_{n}$ converges strongly to  $\varrho$, we  have $\varrho_{1}=\varrho$, and this implies in particular that $\varrho$ satisfies the equation (\ref{heat}). 

\subsection{High compressible limit} \label{conv.}

The last result shows the convergence (in a distribution sense) of a global weak solution $(\varrho_\varepsilon,u_\varepsilon)$ of the system (\ref{cont._1_cap}) - (\ref{mom._1_cap}) to a weak solution $(\varrho,u)$ of (\ref{cont._eps_cap}) - (\ref{mom._eps_cap}), as $\varepsilon\rightarrow0$. For fixed  $\varepsilon > 0$,  we have the existence of weak solutions $(\varrho_\varepsilon,u_\varepsilon)$ satisfying (\ref{weak_sol_reg}) and the energy inequality (\ref{ei_weak}). With a similar procedure used to prove the existence of weak solutions in the case of $\varepsilon=0$ it  is possible to prove the following convergences
\begin{equation} \label{conv-1}
\varrho_{\varepsilon}\rightarrow\varrho\mbox{ in }L^{2/s}\left(0,T;H^{1+s}\left(\Omega\right)\right)\cap C\left(\left[0,T\right];H^{s}\right)\mbox{ for }\mbox{ all }s\in\left(0,1\right),
\end{equation}
\begin{equation} \label{conv-2}
\varrho_{\varepsilon}u_{\varepsilon}\rightharpoonup\varrho u\mbox{ in }L^{2}\left(0,T;L^{2}\left(\Omega\right)\right),
\end{equation}
\begin{equation} \label{conv-3}
\varrho_{\varepsilon}^{2}u_{\varepsilon}\rightharpoonup\varrho^{2}u\mbox{ in }L^{2}\left(0,T;L^{2}\left(\Omega\right)\right),
\end{equation}
\begin{equation} \label{conv-4}
\sqrt{\varrho_{\varepsilon}}u_{\varepsilon}\rightharpoonup\sqrt{\varrho} u\mbox{ in }L^{2}\left(0,T;L^{2}\left(\Omega\right)\right),
\end{equation}
\begin{equation} \label{conv-5}
\varrho_{\varepsilon}^{3/2}u_{\varepsilon}\rightarrow\varrho^{3/2}u\mbox{ in }L^{2}\left(0,T;L^{2}\left(\Omega\right)\right),
\end{equation}
\begin{equation} \label{conv-6}
\varrho_{\varepsilon}^{3/2}\rightarrow\varrho^{3/2}\mbox{ in }L^{2}\left(0,T;L^{2}\left(\Omega\right)\right),
\end{equation}
\begin{equation} \label{conv-7}
\sqrt{\varrho_{\varepsilon}}\nabla\varrho_{\varepsilon}\rightarrow\sqrt{\varrho}\nabla\varrho\mbox{ in }L^{2}\left(0,T;L^{2}\left(\Omega\right)\right),
\end{equation}
\begin{equation} \label{conv-8}
\varrho_{\varepsilon}^{2}\rightarrow\varrho^{2}\mbox{ in }L^{2}\left(0,T;L^{2}\left(\Omega\right)\right),
\end{equation}
\begin{equation} \label{conv-9}
\varrho_{\varepsilon}\nabla\varrho_{\varepsilon}\rightarrow\varrho\nabla\varrho\mbox{ in }L^{2}\left(0,T;L^{2}\left(\Omega\right)\right),
\end{equation}  
\begin{equation} \label{conv-10}
\Delta\varrho_{\varepsilon}\rightharpoonup\Delta\varrho\mbox{ in }L^{2}\left(0,T;L^{2}\left(\Omega\right)\right).
\end{equation}

It is important to stress the fact that there are only two points which deserve some attention. The first one is when we deal with the equivalent of the estimate (\ref{dt_rhou}) for 
$\varepsilon>0$. Indeed in this case, in the weak formulation,  the pressure contributes with a term of the form $\varepsilon\varrho^{\gamma+1}$. The second point is to perform the convergence of the pressure term itself. Therefore,  in order to prove Theorem \ref{th.1}, we need to show a  bound independent by $\varepsilon$ of the quantity $\varrho^{\gamma+1}_\varepsilon$ and 
\begin{equation} \label{convergence}
	\int_{0}^{T}\int_{\Omega}\varepsilon\varrho^{\gamma+1}_\varepsilon\textrm{div}\varphi dxdt\rightarrow0, \ \ \textrm{as} \ \ \varepsilon\rightarrow0.
\end{equation}


\subsubsection{Proof of Theorem \ref{th.1}}
The convergence (\ref{convergence}) can be proved thanks to the bound on the pressure term coming form the energy inequality (\ref{ei_weak}), and the regularity properties of the density coming from Lemma \ref{lm2.}. 
We have
$$
\varepsilon\varrho^{\gamma+1}_\varepsilon=\varepsilon^{1/\gamma}\varrho_\varepsilon\cdot\varepsilon^{\left(\gamma-1\right)/\gamma}\varrho^{\gamma}_\varepsilon.
$$
Consequently, we can write
$$
\int_{0}^{T}\int_{\Omega}\varepsilon\varrho^{\gamma+1}_\varepsilon dxdt\leq\left\Vert \varepsilon^{1/\gamma}\varrho_\varepsilon\right\Vert _{L^{\infty}\left(0,T;L^{\gamma}\right)}\left\Vert \varepsilon^{\left(\gamma-1\right)/\gamma}\varrho^{\gamma}_\varepsilon\right\Vert _{L^{1}\left(0,T;L^{\gamma^{\prime}}\right)} 
$$
$$
\leq C\varepsilon^{1/\gamma^{\prime}}\left\Vert \varrho^{\gamma}_\varepsilon\right\Vert _{L^{1}\left(0,T;L^{\gamma^{\prime}}\right)}
$$
where $\gamma^{\prime}=\gamma/\left(\gamma-1\right)$. The bound for $\varepsilon^{1/\gamma}\varrho_\varepsilon$ in $L^{\infty}\left(0,T;L^{\gamma}\right)$ was obtained from the energy inequality (\ref{ei_weak}). 
Now, we write
$$
\left\Vert \varrho^{\gamma}_\varepsilon\right\Vert _{L^{1}\left(0,T;L^{\gamma^{\prime}}\right)}=\int_{0}^{T}\left\Vert \varrho_\varepsilon\right\Vert _{L^{\gamma^{2}/\left(\gamma-1\right)}}^{\gamma}dt,
$$
and we apply the following interpolation inequality 
$$
\left\Vert \varrho_\varepsilon\right\Vert _{L^{\gamma^{2}/\left(\gamma-1\right)}}\leq\left\Vert \varrho_\varepsilon\right\Vert _{L^{3}}^{\alpha}\left\Vert \varrho_\varepsilon\right\Vert _{L^{\infty}}^{1-\alpha},
$$
with $\alpha=3(\gamma-1)/\gamma^{2}$ and $\gamma>1$. The $L^{\infty}\left(0,T;L^{3}\right)$ bound for $\varrho_\varepsilon$ gives
$$
\int_{0}^{T}\left\Vert \varrho_\varepsilon\right\Vert _{L^{\gamma^{2}/\left(\gamma-1\right)}}^{\gamma}dt\leq C.
$$
As a result, we obtain (\ref{convergence}) and we conclude the proof of Theorem \ref{th.1}.

\section{Quantum fluids} \label{qf}
In this section we deal with the quantum case, namely when $\kappa\left(\varrho\right)=\kappa^2/\varrho$ in (\ref{mom._1}). The system (\ref{cont._1}) - (\ref{mom._1}) reads as
\begin{equation} \label{cont._1_qt}
\partial_{t}\varrho+\textrm{div}\left(\varrho u\right)=0,
\end{equation}
\begin{equation} \label{mom._1_qt}
\partial_{t}\left(\varrho u\right)+\textrm{div}\left(\varrho u\otimes u\right)-2\nu\textrm{div}\left(\varrho D\left(u\right)\right)+\varepsilon\nabla p(\varrho)=2\kappa^2\varrho\nabla\left(\frac{\Delta\sqrt{\varrho}}{\sqrt{\varrho}}\right).
\end{equation}
For fixed $\varepsilon>0$, the quantum Navier-Stokes system have been recently analyzed by several authors. With the same definition of weak solutions used by Bresch et al. \cite{BDL}, Dong \cite{Do}, Jiang \cite{Ji} and J\" ungel \cite{J} proved the global-in-time existence of weak solutions in three-dimension, for large data and under the condition $\gamma>3$. 

Adding a damping term and a singular pressure close to vacuum, Vasseur and Yu \cite{VY-1} and Gisclon and Lacroix-Violet \cite{GLV} proved the global-in-time existence of weak solutions for large data in three-dimension, respectively. 


Antonelli and Spirito \cite{AS_1}, \cite{AS} considered the quantum Navier-Stokes system both in two and in three space dimensions and proved the global existence of finite energy weak solutions for large initial data under some restriction on $\nu$, $\kappa$ and $\gamma$ in a periodic domain $\Omega=\mathbb{T}^{d}$ $(d=2,3)$. 
More precisely, they proved the existence of weak solutions in the following class of regularity
\begin{equation} \label{integr_quantum}
	\left\{ \begin{array}{c}
		\varrho\in L^{\infty}\left(0,T;L^{1}\left(\Omega\right)
		\cap L^{\gamma}\left(\Omega\right)
		\right),\\
		\sqrt{\varrho}u\in L^{\infty}\left(0,T;L^{2}\left(\Omega\right)\right),\\
		\sqrt{\varrho}\in L^{\infty}\left(0,T;H^{1}\left(\Omega\right)\right),
	\end{array}\right.
\end{equation}
satisfying the energy inequality
\begin{equation} \label{e_q_p}
	\frac{d}{dt}\int_{\Omega}\frac{1}{2}\varrho\left|u\right|^{2}
	+\varepsilon\frac{\varrho^{\gamma}}{\gamma-1}
	+2\kappa^{2}\left|\nabla\sqrt{\varrho}\right|^{2}dx\leq0
\end{equation}
with the following conditions on the initial data
\begin{equation} \label{ic_quantum}
\begin{array}{c} \smallskip
\varrho_{0}\geq0 \mbox{ in } \Omega,\ \smallskip
\varrho_{0}\in L^{1}\cap L^{\gamma}\left(\Omega\right),\ \smallskip
\nabla\sqrt{\varrho_{0}}\in L^{2}\cap L^{2+\eta}\left(\Omega\right),\\ \smallskip
u_{0}=0 \mbox{ on } \left\{ \varrho_{0}=0\right\},\ \smallskip
\sqrt{\varrho_{0}}u_{0}\in L^{2}\cap L^{2+\eta}\left(\Omega\right).
\end{array}
\end{equation}

Their result cna be summarized in the  next theorems.
\begin{thm} \label{AS-1}
Let $d=2$. Let $\nu$, $\kappa$ and $\gamma$ positive such that $\kappa<\nu$ and $\gamma>1$, assume that (\ref{ic_quantum}) hold. Then, for any $0<T<\infty$ there exists a finite energy weak solution of the system (\ref{cont._1_qt}) - (\ref{mom._1_qt}) on $\left(0,T\right)\times\mathbb{T}^{2}$, that means a solution satisfying (\ref{integr_quantum}) - \eqref{e_q_p}, with the  continuity equation satisfying the following weak formulation for all $\varphi\in C_{c}^{\infty}\left(\left[0,T\right];C^{\infty}\left(\Omega\right)\right)$ such that $\varphi(T,\cdot)=0$,
\begin{equation} \label{cont_D_q_pressure}
\int_{\Omega}\varrho_{0}\cdot\varphi\left(0,\cdot\right)dx+\int_{0}^{T}\int_{\Omega}\varrho\cdot\partial_{t}\varphi+\int_{0}^{T}\int_{\Omega}\left(\sqrt{\varrho}\sqrt{\varrho}u\right):\nabla\varphi dxdt=0;
\end{equation}
and the momentum equation satisfying the following weak formulation for all $\varphi\in C_{c}^{\infty}\left(\left[0,T\right];C^{\infty}\left(\Omega\right)\right)
	$ such that $\varphi(T,\cdot)=0$,
	$$\int_{\Omega}\varrho_{0}u_{0}\cdot\varphi\left(0,\cdot\right)dx+\int_{0}^{T}\int_{\Omega}\left[\sqrt{\varrho}\left(\sqrt{\varrho}u\right)\cdot\partial_{t}\varphi+\sqrt{\varrho} u\otimes\sqrt{\varrho}:\nabla \varphi
	+\varepsilon\varrho^{\gamma}\textrm{div}\varphi
	\right.
	$$
	$$
	-2\nu\left(\sqrt{\varrho}u\otimes\nabla\sqrt{\varrho}\right):\nabla\varphi-2\nu\left(\nabla\sqrt{\varrho}\otimes\sqrt{\varrho}u\right):\nabla\varphi
	$$
	$$
	+\nu\sqrt{\varrho}\sqrt{\varrho}u\Delta\varphi+\nu\sqrt{\varrho}\sqrt{\varrho}u\nabla\textrm{div}\varphi
	$$
	\begin{equation} \label{wf_mom_q_pressure}
	\left.-4\kappa^{2}\left(\nabla\sqrt{\varrho}\otimes\nabla\sqrt{\varrho}\right):\nabla\varphi+2\kappa^{2}\sqrt{\varrho}\nabla\sqrt{\varrho}\nabla\textrm{div}\varphi\right]dxdt
	=0.
	\end{equation} 
\end{thm}

\begin{thm} \label{AS-2}
	Let $d=3$. Let $\nu$, $\kappa$ and $\gamma$ positive such that $\kappa^2<\nu^2<\frac{9}{8}\kappa^2$ and $1<\gamma<3$, assume that (\ref{ic_quantum}) hold. Then, for any $0<T<\infty$ there exists a finite energy weak solution of the system (\ref{cont._1_qt}) - (\ref{mom._1_qt}) on $\left(0,T\right)\times\mathbb{T}^{3}$, that means a solution satisfying (\ref{integr_quantum}) -\eqref{e_q_p},
	with the  continuity equation satisfying the weak formulation (\ref{cont_D_q_pressure}) and the momentum equation satisfying the weak formulation (\ref{wf_mom_q_pressure}).
\end{thm}

Moreover, recently, the existence of the so-called $\kappa$-entropy solutions was proved for a compressible Navier-Stokes system with degenerate viscosities in the case of singular pressure and drag term (see Bresch et al. \cite{BD-4}). Furthermore, a generalization of the quantum Bohm identity was proposed in \cite{BFPV}, with resulting weak-strong uniqueness analysis for the Navier-Stokes-Korteweg and Euler-Korteweg systems describing quantum fluids (see \cite{BGL}, and \cite{BNV_1}; see also \cite{DEP} for the analysis on the Euler-Korteweg-Poisson system). 

Similar to the analysis on the capillary fluids, our aim is to study the convergence of the weak solutions of the quantum Navier-Stokes system in the limit as $\varepsilon\rightarrow0$, assuming the existence of the weak solutions for the pressureless system. Note that this last assumption is based on the fact that we leave open the problem of the existence of the weak solutions for the pressureless quantum Navier-Stokes system that will be a matter of future research as mentioned in the end of this paper. Moreover, we perform a weak strong-uniqueness analysis for the pressureless system assuming the existence of the strong solution. Weak-strong uniqueness means that a weak and strong solution emanating from the same initial data coincide as long as the latter exists.
This last analysis will follow the recent result of Bresch et al. \cite{BGL} (see also  \cite{BNV_1}) in which an "augmented" version of the quantum Navier-Stokes system combined with a relative energy inequality approach (see for example Feireisl et al. \cite{FSN}, \cite{FJN}) is used. The authors would like to specify that the weak-strong uniqueness analysis was not performed in the capillary case for the following reasons. First, the recent theory of of Bresch et al. \cite{BGL} does not apply to the capillary case, namely for $\kappa(\varrho)=\kappa$. Second, the usual analyzes do not work because of the density dependent viscosity. Indeed, as it will be stressed later in Section \ref{w-s}, the $H^{1}$ bound for the velocity is no longer available and this does not allow to manipulate the viscosity terms in a standard sense through the use of the Korn's inequality (see for example Feireisl et al. \cite{FJN}). 
 
\subsection{Weak-solutions and main results} \label{ws-q}
For $\varepsilon=0$, the system (\ref{cont._1_qt}) - (\ref{mom._1_qt}) reads as
\begin{equation} \label{cont._eps_qt}
\partial_{t}\varrho+\textrm{div}\left(\varrho u\right)=0,
\end{equation}
\begin{equation} \label{mom._eps_qt}
\partial_{t}\left(\varrho u\right)+\textrm{div}\left(\varrho u\otimes u\right)-2\nu\textrm{div}\left(\varrho D\left(u\right)\right)=2\kappa^2\varrho\nabla\left(\frac{\Delta\sqrt{\varrho}}{\sqrt{\varrho}}\right).
\end{equation} 
Now, in the spirit of \cite{AS} we define the weak solutions of the pressureless quantum Navier-Stokes system (\ref{cont._eps_qt}) - (\ref{mom._eps_qt}).
\begin{defn} \label{ws_quantum}
	We say that $(\varrho,u)$ is a weak solution of (\ref{cont._eps_qt}) - (\ref{mom._eps_qt}) on $(0,T)$ if and only if
	\begin{equation} \label{integr}
	\left\{ \begin{array}{c}
	\sqrt{\varrho}\in L^{2}\left(0,T;L^{2}\left(\Omega\right)
	\right),\\
	\sqrt{\varrho}u, \ \nabla \sqrt{\varrho}\in L^{2}\left(0,T;L^{2}\left(\Omega\right)\right);
	\end{array}\right.
	\end{equation}
	the continuity equation satisfies the following weak formulation for all $\varphi\in C_{c}^{\infty}\left(\left[0,T\right];C^{\infty}\left(\Omega\right)\right)
	$ such that $\varphi(T,\cdot)=0$,
	\begin{equation} \label{cont_D_q}
	\int_{\Omega}\varrho_{0}\cdot\varphi\left(0,\cdot\right)dx+\int_{0}^{T}\int_{\Omega}\varrho\cdot\partial_{t}\varphi+\int_{0}^{T}\int_{\Omega}\left(\sqrt{\varrho}\sqrt{\varrho}u\right):\nabla\varphi dxdt=0;
	\end{equation}
	the momentum equation satisfies the following weak formulation for all $\varphi\in C_{c}^{\infty}\left(\left[0,T\right];C^{\infty}\left(\Omega\right)\right)
	$ such that $\varphi(T,\cdot)=0$,
	$$\int_{\Omega}\varrho_{0}u_{0}\cdot\varphi\left(0,\cdot\right)dx+\int_{0}^{T}\int_{\Omega}\left[\sqrt{\varrho}\left(\sqrt{\varrho}u\right)\cdot\partial_{t}\varphi+\sqrt{\varrho} u\otimes\sqrt{\varrho}:\nabla \varphi
	\right.
	$$
	$$
	-2\nu\left(\sqrt{\varrho}u\otimes\nabla\sqrt{\varrho}\right):\nabla\varphi-2\nu\left(\nabla\sqrt{\varrho}\otimes\sqrt{\varrho}u\right):\nabla\varphi
	$$
	$$
	+\nu\sqrt{\varrho}\sqrt{\varrho}u\Delta\varphi+\nu\sqrt{\varrho}\sqrt{\varrho}u\nabla\textrm{div}\varphi
	$$
	\begin{equation} \label{wf_mom_q}
	\left.-4\kappa^{2}\left(\nabla\sqrt{\varrho}\otimes\nabla\sqrt{\varrho}\right):\nabla\varphi+2\kappa^{2}\sqrt{\varrho}\nabla\sqrt{\varrho}\nabla\textrm{div}\varphi\right]dxdt
	=0;
	\end{equation}
    there exists  $\mathcal{S}\in L^2((0,T)\times\Omega)$ such that $\sqrt{\varrho}\mathcal{S}=\mbox{Symm}(\nabla(\varrho u)-2\nabla\sqrt{\varrho}\otimes\sqrt{\varrho} u)$ in $\mathcal{D}'$, satisfying the following energy inequality
    $$
    \sup_{t\in(0,T)}\int_{\Omega}\frac{1}{2}\left(\rho |u|^{2}
    +\kappa\left|\nabla\sqrt{\varrho}\right|^{2}\right)
    +2\nu\int_{0}^T\int_{\Omega}\left|\mathcal{S}\right|^2dxdt
    $$
    \begin{equation} \label{ee_q}
    \leq \int_{\Omega}\frac{1}{2}\left(\varrho_{0}\left|u_{0}\right|^{2}+\kappa\left|\nabla\sqrt{\varrho_{0}}\right|^{2}\right)dx;
    \end{equation}
\end{defn}

\bigskip
Our main result is the following.

\begin{thm} \label{th_conv_q}
	Let $\Omega=\mathbb{T}^{d}$ ($d=2$ or 3) be a periodic domain. Assume that there exists a weak solutions $(\varrho,u)$ of the quantum Navier-Stokes system (\ref{cont._eps_qt}) - (\ref{mom._eps_qt}). Then, as $\varepsilon\rightarrow0$, a global weak solution $(\varrho_\varepsilon,u_\varepsilon)$ of the system (\ref{cont._1_qt}) - (\ref{mom._1_qt}) converges (in a distribution sense) to a weak solution $(\varrho,u)$, in the sense of the Definition \ref{ws_quantum},  of the system (\ref{cont._eps_qt}) - (\ref{mom._eps_qt}). 
\end{thm}
  
\subsection{High compressible limit} \label{conv_q}
In this section we show the convergence (in a distribution sense) of a weak solution $(\varrho_\varepsilon,u_\varepsilon)$ of the system (\ref{cont._1_qt}) - (\ref{mom._1_qt}) to a weak solution $(\varrho,u)$ of the pressureless system (\ref{cont._eps_qt}) - (\ref{mom._eps_qt}), as $\varepsilon\rightarrow0$.
For fixed  $\varepsilon > 0$,  we have the existence of weak solutions $(\varrho_\varepsilon,u_\varepsilon)$ satisfying (\ref{integr_quantum}) and the energy inequality (\ref{e_q_p}). Thanks to the stability analysis provided by Antonelli and Spirito in \cite{AS_1} it is possible to prove the following convergences
\begin{equation} \label{conv-q-1}
\sqrt{\varrho_{\varepsilon}}\rightarrow\sqrt{\varrho}\mbox{ in }L^{2}\left(0,T;H^{1}\left(\Omega\right)\right),
\end{equation}
\begin{equation} \label{conv-q-3}
\sqrt{\varrho_{\varepsilon}}u_{\varepsilon}\rightarrow\sqrt{\varrho}u\mbox{ in }L^{2}\left(0,T;L^{2}\left(\Omega\right)\right).
\end{equation}
In order to justify the convergences (\ref{conv-q-1}) and (\ref{conv-q-3}), we report the compactness analysis of Antonelli and Spirito \cite{AS_1} in the Appendix. Potentially, this analysis is required to prove the existence of the weak solutions for the pressureless quantum Navier-Stokes system. However, as mentioned before and in the conclusions, for the time being this is beyond the scope of the present work. 

Consequently, in order to prove Theorem \ref{th_conv_q}, the only term in the weak formulation (\ref{wf_mom_q_pressure}) that requires some attention is the pressure contribution. Namely, we need to show that
\begin{equation} \label{convergence-1}
\int_{0}^{T}\int_{\Omega}\varepsilon\varrho^{\gamma}_\varepsilon\textrm{div}\varphi dxdt\rightarrow0, \ \ \textrm{as} \ \ \varepsilon\rightarrow0.
\end{equation}
This means that we need to find a uniform bound, independent by $\varepsilon$, of the quantity $\varrho^{\gamma}_\varepsilon$ and take the $\varepsilon$-limit in the weak formulation (\ref{wf_mom_q_pressure}). 

\subsubsection{Proof of Theorem \ref{th_conv_q}}

The convergence (\ref{convergence-1}) can be proved thanks to the regularity properties of the density. 
%
Indeed, thanks to the following interpolation inequality 
$$
\left\Vert \varrho_\varepsilon\right\Vert _{L^{\gamma}}\leq\left\Vert \varrho_\varepsilon\right\Vert _{L^{3}}^{\alpha}\left\Vert \varrho_\varepsilon\right\Vert _{L^{1}}^{1-\alpha},
$$
with $\alpha=\frac{3}{2}(\gamma-1)/\gamma$ and $1,\gamma<3$, we have
$$
\int_{0}^{T}\left\Vert \varrho_\varepsilon\right\Vert _{L^{\gamma}}^{\gamma}dt\leq C.
$$ Here, the $L^{\infty}\left(0,T;L^{3}\right)$ bound for $\varrho_\varepsilon$ is obtained from the $H^1$ bound for $\sqrt{\varrho_\varepsilon}$.  
%
%
As a consequence we have (\ref{convergence-1}) and this ends the proof of Theorem \ref{th_conv_q}.

\subsection{Weak-strong uniqueness in the pressureless case} \label{w-s}
In this section, we would like to perform the weak-strong uniqueness analysis. As mentioned before, the analysis requires to introduce an "augmented" version of the pressureless quantum Navier-Stokes system. The reason for that stands in the fact that a $H^1$ bound for the velocity is no longer available because of the density dependent viscosity. Consequently, standard application of the Korn's inequality in the weak-strong uniqueness context is not possible (see for example Feireisl et al. \cite{FJN}). Moreover, the presence of the quantum Bohm potential adds some difficulties that seem not possible to overcome by the methodology proposed in Bresch et al.  \cite{BNV_1}. 

\subsubsection{"Augmented" version of the quantum Navier-Stokes system} \label{augm_qns}
In order to derive an augmented version of the quantum Navier-Stokes system, we multiply the continuity equation by $\mu^{\prime}\left(\varrho\right)$ and we write the corresponding equation for $\mu\left(\varrho\right)$, namely
\begin{equation}\label{mu_eq}
\partial_{t}\mu\left(\varrho\right)+\textrm{div}_{x}\left(\mu\left(\varrho\right)u\right)=0.
\end{equation}
The following computation keeps $\mu\left(\varrho\right)$ as a general function of the density. Later, we will consider our case in which $\mu\left(\varrho\right)=\nu\varrho$. We differentiate respect to space, and observing $\nabla\mu\left(\varrho\right)=\varrho\nabla\phi\left(\varrho\right)$, we obtain
\begin{equation}\label{phi_eq}
\partial_{t}\left(\varrho\nabla\phi\left(\varrho\right)\right)+\textrm{div}_{x}\left(\varrho u\otimes\nabla\phi\left(\varrho\right)\right)+\textrm{div}_{x}\left(\mu\left(\varrho\right)\nabla^t u\right)=0.
\end{equation}
Now, we introduce the following vector field (see for example Bresch et al. \cite{BD-4})
\begin{equation} \label{w}
w=u+\nabla\phi\left(\varrho\right),\qquad \textrm{div}_{x}w=0.
\end{equation}
Using the momentum equation \eqref{mom._eps_qt}
we have
\begin{equation} \label{mom_ag}
\partial_{t}\left(\varrho w\right)+\textrm{div}_{x}\left(\varrho w\otimes u\right)-2\textrm{div}_{x}\left(\mu\left(\varrho\right)D\left(u\right)\right)+\textrm{div}_{x}\left(\mu\left(\varrho\right)\nabla^t u\right)=2\kappa^{2}\varrho\nabla\left(\frac{\Delta\sqrt{\varrho}}{\sqrt{\varrho}}\right).
\end{equation}
Consequently, we have
\begin{equation} \label{mom_ag_1}
\partial_{t}\left(\varrho w\right)+\textrm{div}_{x}\left(\varrho w\otimes u\right)-\textrm{div}_{x}\left(\mu\left(\varrho\right)D\left(u\right)\right)-\textrm{div}_{x}\left(\mu\left(\varrho\right)A\left(u\right)\right)=2\kappa^{2}\varrho\nabla\left(\frac{\Delta\sqrt{\varrho}}{\sqrt{\varrho}}\right),
\end{equation}
with $A(u)=(\nabla u-\nabla^{t}u)/2$, that can be rewritten as follows
$$
\partial_{t}\left(\varrho w\right)+\textrm{div}_{x}\left(\varrho w\otimes u\right)-\textrm{div}_{x}\left(\mu\left(\varrho\right)D\left(w\right)\right)-\textrm{div}_{x}\left(\mu\left(\varrho\right)A\left(w\right)\right)
$$
\begin{equation} \label{mom_ag_3}
+\textrm{div}_{x}\left(\mu\left(\varrho\right)\nabla^{2}\phi\left(\varrho\right)\right)=2\kappa^{2}\varrho\nabla\left(\frac{\Delta\sqrt{\varrho}}{\sqrt{\varrho}}\right).
\end{equation}
Now, in our case $\mu\left(\varrho\right)=\nu\varrho$. Observing that
$$
\nabla\phi\left(\varrho\right)=\phi^{\prime}\left(\varrho\right)\nabla\varrho=\frac{\mu^{\prime}\left(\varrho\right)}{\varrho}\nabla\varrho=\nu\nabla\log\varrho,
$$
$$
2\kappa^{2}\varrho\nabla\left(\frac{\Delta\sqrt{\varrho}}{\sqrt{\varrho}}\right)=\kappa^{2}\textrm{div}_{x}\left(\varrho\nabla\nabla\log\varrho\right)=\frac{\kappa^{2}}{\nu}\textrm{div}_{x}\left(\varrho\nabla v\right),
$$
we can rewrite (\ref{mom_ag_3}) as
\begin{equation} \label{mom_ag_4}
\partial_{t}\left(\varrho w\right)+\textrm{div}_{x}\left(\varrho w\otimes u\right)-\nu\textrm{div}_{x}\left(\varrho\nabla w\right)+\left(\nu-\frac{\kappa^{2}}{\nu}\right)\textrm{div}_{x}\left(\varrho\nabla v\right)=0,
\end{equation}
where $v=\nu\nabla\log\varrho$.
The relation above suggests to write an equation for $v$. Then, from relation (\ref{phi_eq}), we have
\begin{equation} \label{v}
\partial_{t}\left(\varrho v\right)+\textrm{div}_{x}\left(\varrho v\otimes u\right)+\nu\textrm{div}_{x}\left(\varrho\nabla u\right)=0.
\end{equation}
Consequently, the so-called "augmented" version of the pressureless quantum Navier-Stokes system reads
\begin{equation}\label{ag_1}
\partial_{t}\varrho+\textrm{div}_{x}\left(\varrho u\right)=0,
\end{equation}
\begin{equation}\label{ag_2}
\partial_{t}\left(\varrho w\right)+\textrm{div}_{x}\left(\varrho w\otimes u\right)-\nu\textrm{div}_{x}\left(\varrho\nabla w\right)-\left(\frac{\kappa^{2}}{\nu}-\nu\right)\textrm{div}_{x}\left(\varrho\nabla v\right)=0,
\end{equation}
\begin{equation}\label{ag_3}
\partial_{t}\left(\varrho v\right)+\textrm{div}_{x}\left(\varrho v\otimes u\right)+\nu\textrm{div}_{x}\left(\varrho\nabla u\right)=0.
\end{equation}
Now, using the definition of $v$, we rewrite the above system as follows,
\begin{equation}\label{ag1}
	\partial_{t}\varrho+\textrm{div}_{x}\left(\varrho u\right)=0,
\end{equation}
\begin{equation}\label{ag2}
	\partial_{t}\left(\varrho w\right)+\textrm{div}_{x}\left(\varrho w\otimes u\right)-\nu\textrm{div}_{x}\left(\varrho\nabla w\right)-\left(\kappa^{2}-\nu^{2}\right)\textrm{div}_{x}\left(\varrho\nabla v\right)=0,
\end{equation}
\begin{equation}\label{ag3}
	\partial_{t}\left(\varrho v\right)+\textrm{div}_{x}\left(\varrho v\otimes u\right)+\textrm{div}_{x}\left(\varrho\nabla u\right)=0.
\end{equation}

\begin{rem} \label{v_rem}
	Note that we passed from $v=\nu\nabla\log\varrho$ to $v=\nabla\log\varrho$ using the same notation, and we simplified the viscosity $\nu$ in the last equation.
\end{rem}


Now, the idea is to pass at the new unknown $\overline{v}=v\sqrt{\kappa^{2}-\nu^{2}}$, with $\kappa>\nu$, and to define a global weak solutions of the system (\ref{ag1}) - (\ref{ag3}) as follows (see Bresch et al. \cite{BGL})
\begin{defn} \label{ws_augm}
	We say that $(\varrho,\overline{v}, w)$ is a weak solution of the "augmented" system (\ref{ag1}) - (\ref{ag3}) if satisfies the following system in the distribution sense
	\begin{equation} \label{as_cont}
	\partial_{t}\varrho+\textrm{div}\left(\varrho u\right)=0,
	\end{equation}
	\begin{equation} \label{as_mom_w}
	\partial_{t}\left(\varrho w\right)+\textrm{div}\left(\varrho w\otimes u\right)-\nu\textrm{div}\left(\sqrt{\varrho}\mathbb{T}\left(w\right)\right)-\sqrt{\kappa^{2}-\nu^{2}}\textrm{div}\left(\sqrt{\varrho}\mathbb{T}\left(\overline{v}\right)\right)=0,
	\end{equation}
	\begin{equation} \label{as_mom_v}
	\partial_{t}\left(\varrho\overline{v}\right)+\textrm{div}\left(\varrho\overline{v}\otimes u\right)-\textrm{div}\left(\sqrt{\varrho}\mathbb{T}\left(\overline{v}\right)\right)+\sqrt{\kappa^{2}-\nu^{2}}\textrm{div}\left(\sqrt{\varrho}\mathbb{T}\left(w\right)\right)=0,
	\end{equation}
	with
	\begin{equation} \label{T}
	\sqrt{\varrho}\mathbb{T}\left(\vartheta\right)=\nabla\left(\varrho\vartheta\right)-\frac{1}{\sqrt{\kappa^{2}-\nu^{2}}}\varrho\vartheta\otimes\overline{v}, \ \ (\vartheta=w,\overline{v}),
	\end{equation}
	and endowed with the following energy inequality
	\begin{equation} \label{as_ei}
	\frac{d}{dt}\frac{1}{2}\int_{\Omega}\varrho\left|w\right|^{2}+\varrho\left|\overline{v}\right|^{2}dx+\nu\int_{\Omega}\left|\mathbb{T}\left(w\right)\right|^{2}+\left|\mathbb{T}\left(\overline{v}\right)\right|^{2}dx\leq0.
	\end{equation}
\end{defn}

\begin{rem} \label{weak_consist}
	We would like to stress that a global weak solution of the pressureless quantum Navier-Stokes system is also a global weak solution of the augmented version. Indeed, as remarked by Bresch et al. \cite{BGL}, the following equation, for example, is satisfied in the distribution sense
	\begin{equation} \label{cont_mu}
	\nu\left[\partial_{t}\mu\left(\varrho\right)+\textrm{div}\left(\mu\left(\varrho\right)u\right)\right]=0
	\end{equation}  
	with $u=w-\nu\overline{v}\sqrt{\kappa^{2}-\nu^{2}}$. Differentiating (\ref{cont_mu}), we have
	\[
	\nu\left[\partial_{t}\nabla\mu\left(\varrho\right)+\textrm{div}\left(\nabla\left(\mu\left(\varrho\right)u\right)\right)\right]=0.
	\]
	Consequently, by the definition of $v$ and $\sqrt{\varrho}\mathbb{T}\left(u \right)$, we can write
	\[
	\nu\left[\partial_{t}\left(\varrho v\right)+\textrm{div}\left(\varrho v\otimes u\right)+\textrm{div}\left(\sqrt{\varrho}\mathbb{T}\left(u\right)\right)\right]=0,
	\]
	that explains the assertion above.
\end{rem}

Now, we are ready to state our weak strong uniqueness result.

\begin{thm} \label{th_ws_q}
	Let $\Omega=\mathbb{T}^{d}$ ($d=2$ or 3) be a periodic domain. Let us consider $(\varrho,u)$ be a weak solutions to the pressureless quantum Navier-Stokes system, and $(r,U)$ be the corresponding strong solution emanating from the same initial data. Then, $(u, \overline{v}, w)=(U, \overline{V}, W)$ or $(\varrho,u)=(r,U)$ in $(0,T)\times\Omega$, which corresponds to a weak-strong uniqueness property.
\end{thm}

\subsubsection{Proof of Theorem \ref{th_ws_q}}
As in Bresch et al. \cite{BGL}, we define a relative energy functional for the system (\ref{as_cont}) - (\ref{as_mom_v}) by the following relation 
\[
\mathcal{E}\left(\overline{v},w|\overline{V},W\right)=\frac{1}{2}\int_{\Omega}\varrho\left(\left|w-W\right|^{2}+\left|\overline{v}-\overline{V}\right|^{2}\right)dx
\]
\begin{equation} \label{rel_en_as}
+\nu\int_{0}^{T}\int_{\Omega}\varrho\left(\left|\frac{\mathbb{T}\left(\overline{v}\right)}{\sqrt{\varrho}}-\nabla\overline{V}\right|^{2}+\left|\frac{\mathbb{T}\left(w\right)}{\sqrt{\varrho}}-\nabla W\right|^{2}\right)dxdt.
\end{equation}
Note that, a viscous part is present in comparison with the usual definition of the relative energy functional (see for example Feireisl et al. \cite{FJN}). As shown below this, together with the augmented system and the respective system for strong solutions, will give us the possibility to treat the viscous terms during the weak-strong uniqueness analysis.

Now, the following Proposition, proving a relative energy inequality, holds.
\begin{prop} \label{REI}
	Let
	\[
	r\in C^{1}\left(\left[0,T\right]\times\overline{\Omega}\right), \ \ r>0, \ \ \overline{V}, \ W\in C^{2}\left(\left[0,T\right]\times\overline{\Omega}\right).
	\]
	Then, any global weak solutions $\left(\varrho,\overline{v},w\right)$ of the system (\ref{as_cont}) - (\ref{as_mom_v}) satisfies the following relative energy inequality
	\[
	\mathcal{E}\left(t\right)\leq\mathcal{E}\left(0\right)
	\]
	\[
	+\int_{0}^{T}\int_{\Omega}\varrho\left(\partial_{t}\overline{V}\cdot\left(\overline{V}-\overline{v}\right)+\left(\nabla\overline{V}u\right)\cdot\left(\overline{V}-\overline{v}\right)\right)dxdt
	\]
	\[
	+\int_{0}^{T}\int_{\Omega}\varrho\left(\partial_{t}W\cdot\left(W-w\right)+\left(\nabla Wu\right)\cdot\left(W-w\right)\right)dxdt
	\]
	\[
	+\nu\int_{0}^{T}\int_{\Omega}\varrho\left(\left|\nabla\overline{V}\right|^{2}+\left|\nabla W\right|^{2}\right)-\sqrt{\varrho}\left(\mathbb{T}\left(\overline{v}\right):\nabla\overline{V}+\mathbb{T}\left(w\right):\nabla W\right)dxdt
	\]
	\begin{equation} \label{rei_as}
	+\sqrt{\kappa^{2}-\nu^{2}}\int_{0}^{T}\int_{\Omega}\sqrt{\varrho}\left(\mathbb{T}\left(\overline{v}\right):\nabla W-\mathbb{T}\left(w\right):\nabla\overline{V}\right)dxdt.
	\end{equation}
\end{prop}
\begin{proof}
	The proof follows similar arguments as in Bresch et al. \cite{BGL}, Preposition 14, in which is possible to drop the pressure contribution.
\end{proof}

Now, we introduce the following system satisfied by the strong solutions $\left(r,\overline{V},W\right)$:
\begin{equation} \label{cont_strong}
\partial_{t}r+\textrm{div}\left(rU\right)=0,
\end{equation}
\begin{equation} \label{mom_strong_w}
r\left(\partial_{t}W+\nabla W\cdot U\right)-\nu\textrm{div}\left(r\nabla W\right)-\sqrt{\kappa^{2}-\nu^{2}}\textrm{div}\left(r\nabla\overline{V}\right)=0,
\end{equation}
\begin{equation} \label{mom_strong_v}
r\left(\partial_{t}\overline{V}+\nabla\overline{V}\cdot U\right)-\nu\textrm{div}\left(r\nabla\overline{V}\right)+\sqrt{\kappa^{2}-\nu^{2}}\textrm{div}\left(r\nabla W\right)=0,
\end{equation}
with
\begin{equation} \label{strong_UVW}
U=W-\nu V, \ \ \overline{V}=\sqrt{\kappa^{2}-\nu^{2}}V,
\end{equation}
and belonging to the following class
\begin{equation} \label{reg_class}
\begin{array}{c}
0<\inf_{\left(0,T\right)\times\Omega}r\leq r\leq\sup_{\left(0,T\right)\times\Omega}r<+\infty,\\
\nabla r\in L^{2}\left(0,T;L^{\infty}\left(\Omega\right)\right)\cap L^{1}\left(0,T;W^{1,\infty}\left(\Omega\right)\right),\\
W\in L^{\infty}\left(0,T;W^{2,\infty}\left(\Omega\right)\right)\cap W^{1,\infty}\left(0,T;L^{\infty}\left(\Omega\right)\right),\\
\overline{V}\in L^{\infty}\left(0,T;W^{2,\infty}\left(\Omega\right)\right)\cap W^{1,\infty}\left(0,T;L^{\infty}\left(\Omega\right)\right).
\end{array}
\end{equation}
The following Proposition holds.
\begin{prop}
	Let $\left(r,\overline{V},W\right)$ be a strong solution of (\ref{cont_strong}) - (\ref{mom_strong_v}). Then, any weak solution of the system (\ref{as_cont}) - (\ref{as_mom_v}) satisfies the following inequality
	\[
	\mathcal{E}\left(t\right)\leq \mathcal{E}\left(0\right)
	\]
	\[
	+\int_{0}^{T}\int_{\Omega}\varrho\left[\left(\nabla\overline{V}\cdot\left(u-U\right)\right)\cdot\left(\overline{V}-\overline{v}\right)+\left(\nabla W\cdot\left(u-U\right)\right)\cdot\left(W-w\right)\right]dxdt
	\]
	\[
	+\nu\int_{0}^{T}\int_{\Omega}\varrho\left(\left|\nabla\overline{V}\right|^{2}+\left|\nabla W\right|^{2}\right)dxdt
	\]
	\[
	-\nu\int_{0}^{T}\int_{\Omega}\sqrt{\varrho}\left(\mathbb{T}\left(\overline{v}\right):\nabla\overline{V}+\mathbb{T}\left(w\right):\nabla W\right)dxdt
	\]
	\[
	+\nu\int_{0}^{T}\int_{\Omega}\frac{\varrho}{r}\left(\textrm{div}\left(r\nabla\overline{V}\right)\cdot\left(\overline{V}-\overline{v}\right)+\textrm{div}\left(r\nabla W\right)\cdot\left(W-w\right)\right)dxdt
	\]
	\[
	+\sqrt{\kappa^{2}-\nu^{2}}\int_{0}^{T}\int_{\Omega}\sqrt{\varrho}\left(\mathbb{T}\left(\overline{v}\right):\nabla W-\mathbb{T}\left(w\right):\nabla\overline{V}\right)dxdt
	\]
	\begin{equation} \label{rei_as_1}
	+\sqrt{\kappa^{2}-\nu^{2}}\int_{0}^{T}\int_{\Omega}\frac{\varrho}{r}\left(\textrm{div}\left(r\nabla\overline{V}\right)\cdot\left(W-w\right)-\textrm{div}\left(r\nabla W\right)\cdot\left(\overline{V}-\overline{v}\right)\right)dxdt.
	\end{equation}	
\end{prop}
\begin{proof}
	The proof follows similar arguments as in Bresch et al. \cite{BGL}, Preposition 15, in which is possible to drop the pressure contribution.
\end{proof}
We rewrite the relation (\ref{rei_as_1}) as follows
\[
\mathcal{E}\left(t\right)\leq \mathcal{E}\left(0\right)
\]
\[
+\int_{0}^{T}\int_{\Omega}\varrho\left[\left(\nabla\overline{V}\cdot\left(u-U\right)\right)\cdot\left(\overline{V}-\overline{v}\right)+\left(\nabla W\cdot\left(u-U\right)\right)\cdot\left(W-w\right)\right]dxdt
\]
\begin{equation} \label{rei_as_2}
+\nu I_{1} + \sqrt{\kappa^{2}-\nu^{2}} I_{2},
\end{equation}
where
\[
I_{1}=\int_{0}^{T}\int_{\Omega}\varrho\left(\left|\nabla\overline{V}\right|^{2}+\left|\nabla W\right|^{2}\right)dxdt
\]
\[
-\int_{0}^{T}\int_{\Omega}\sqrt{\varrho}\left(\mathbb{T}\left(\overline{v}\right):\nabla\overline{V}+\mathbb{T}\left(w\right):\nabla W\right)dxdt
\]
\[
+\int_{0}^{T}\int_{\Omega}\frac{\varrho}{r}\left(\textrm{div}\left(r\nabla\overline{V}\right)\cdot\left(\overline{V}-\overline{v}\right)+\textrm{div}\left(r\nabla W\right)\cdot\left(W-w\right)\right)dxdt
\]
and
\[
I_{2}=\int_{0}^{T}\int_{\Omega}\frac{\varrho}{r}\left(\textrm{div}\left(r\nabla\overline{V}\right)\cdot\left(W-w\right)-\textrm{div}\left(r\nabla W\right)\cdot\left(\overline{V}-\overline{v}\right)\right)dxdt
\]
\[
+\int_{0}^{T}\int_{\Omega}\sqrt{\varrho}\left(\mathbb{T}\left(\overline{v}\right):\nabla W-\mathbb{T}\left(w\right):\nabla\overline{V}\right)dxdt.
\]
Now, for $\mu\left(\varrho\right)=\nu\varrho$ and $\mu\left(r \right)=\nu r$, is possible to show that
\[
-\int_{0}^{T}\int_{\Omega}\sqrt{\varrho}\left(\mathbb{T}\left(\overline{v}\right):\nabla\overline{V}+\mathbb{T}\left(w\right):\nabla W\right)dxdt
\]
\[
+\int_{0}^{T}\int_{\Omega}\frac{\varrho}{r}\left(\textrm{div}\left(r\nabla\overline{V}\right)\cdot\left(\overline{V}-\overline{v}\right)+\textrm{div}\left(r\nabla W\right)\cdot\left(W-w\right)\right)dxdt,
\]
\begin{equation} \label{I_1}
=-\int_{0}^{T}\int_{\Omega}\varrho\left(v-V\right)\left(\nabla\overline{V}\left(\overline{V}-\overline{v}\right)+\nabla W\left(W-w\right)\right)dxdt,
\end{equation}
and
\begin{equation} \label{I_2}
I_{2}=-\int_{0}^{T}\int_{\Omega}\varrho\left(v-V\right)\left(\nabla\overline{V}\left(W-w\right)+\nabla W\left(\overline{V}-\overline{v}\right)\right)dxdt.
\end{equation}
Indeed, realizing that $\mu\left(\varrho\right)/\varrho=v$ and  $\mu\left(r\right)/r=V$, relation (\ref{I_1}) and (\ref{I_2}) can be obtained using the definition (\ref{T}) and integrating by parts. Consequently, thanks to the regularity properties (\ref{reg_class}), we have
\begin{equation} \label{I_1_I_2}
\left|I_{1}+I_{2}\right|\leq C\int_{0}^{T}\int_{\Omega}\varrho\left(\left|v-V\right|^{2}+\left|\overline{v}-\overline{V}\right|^{2}+\left|w-W\right|^{2}\right)dxdt.
\end{equation}
Finally, from relation (\ref{rei_as_2}), we end up with
\[
\mathcal{E}\left(t\right)-\mathcal{E}\left(0\right)\leq C\int_{0}^{T}\int_{\Omega}\varrho\left(\left|u-U\right|^{2}+\left|\overline{v}-\overline{V}\right|^{2}+\left|w-W\right|^{2}\right)dxdt,
\]
that means
\begin{equation} \label{gronwall}
\mathcal{E}\left(t\right) \leq \mathcal{E}\left(0\right) + C\int_{0}^{T}\mathcal{E}\left(t\right).
\end{equation}
In the case weak and strong solutions are emanated from the same initial data, then $\mathcal{E}\left(0\right)=0$ and we obtain $w=W$, $u=U$ and $\overline{v}=\overline{V}$ thanks to the application of the Gronwall Lemma. The weak-strong uniqueness for the densities $\varrho$ and $r$ follows by recalling the uniqueness properties of the continuity equations (\ref{ag_1}) or (\ref{cont_strong}) with smooth velocity field $u=U$ and initial data $\varrho_{0}=r_{0}$.
%

\section{Conclusions} \label{end}
The present work focused on the analysis of the highly compressible limit for an isothermal model of capillary and quantum fluid and on the investigation of the related pressureless systems.

In the first case, we prove the existence of weak solution for the pressureless system and the convergence results. The existence was obtained with the notion of weak solutions presented in Bresch et al. \cite{BDL}. It will be a matter of future research to prove the existence of weak solutions following the notion of weak solutions introduced by Antonelli and Spirito \cite{AS_3} (see also \cite{AS_2}). In the same framework, also a weak-strong uniqueness analysis could be of some interest.

In the second case the convergence towards the pressureless system is shown and a weak-strong uniqueness analysis is developed. A relative entropy inequality based on an "augmented" version of the Navier-Stokes system was used and the resulting analysis on weak-strong uniqueness property was performed assuming the existence of the strong solution. 
In the same spirit, also the existence of weak solutions for the pressureless quantum Navier-Stokes system could be a matter of future investigations.

\appendix
\section{Appendix} \label{app}
The following analysis aims to prove a stability result for a quantum Navier-Stokes system without the presence of the pressure contribution. The analysis follows a similar line of \cite{AS_1} and we sketch here some steps in order to justify the convergences (\ref{conv-q-1}) - (\ref{conv-q-3}). 

We consider the quantum Navier-Stokes system (\ref{cont._1_qt}) - (\ref{mom._1_qt}) where we set $\varepsilon=0$, namely

\begin{equation} \label{cont._1_qt_press}
\partial_{t}\varrho+\textrm{div}\left(\varrho u\right)=0,
\end{equation}
\begin{equation} \label{mom._1_qt_press}
\partial_{t}\left(\varrho u\right)+\textrm{div}\left(\varrho u\otimes u\right)-2\nu\textrm{div}\left(\varrho D\left(u\right)\right)=2\kappa^2\varrho\nabla\left(\frac{\Delta\sqrt{\varrho}}{\sqrt{\varrho}}\right).
\end{equation}

We start with the \textit{a priori} estimates needed for the convergence. The first lemma is the standard energy inequality associated to the quantum Navier-Stokes system (\ref{cont._1_qt_press}) - (\ref{mom._1_qt_press}).
\begin{lemma} \label{ei_press}
	Let $\left(\varrho_{n},u_{n}\right)$ be a smooth solution of (\ref{cont._1_qt_press}) - (\ref{mom._1_qt_press}). Then,
	\begin{equation} \label{ee_press}
	\frac{d}{dt}\left(\int_{\Omega}\varrho_{n}\frac{\left|u_{n}\right|^{2}}{2}+2\kappa^{2}\left|\nabla\sqrt{\varrho_{n}}\right|^{2}dx\right)+2\nu\int_{\Omega}\varrho_{n}\left|Du_{n}\right|^{2}dx=0.
	\end{equation}
\end{lemma}
\begin{proof}
	See J\" ungel \cite{J}.
\end{proof}
Now, we are going to consider a velocity field $w=u+\mu\nabla\log\varrho$. Consequently, (\ref{cont._1_qt_press}) - (\ref{mom._1_qt_press}) can be transformed in terms of $(\varrho,w)$. More precisely, we have the following lemma.  
\begin{lemma} \label{eq_w}
	Let $\left(\varrho_{n},u_{n}\right)$ be a smooth solution of (\ref{cont._1_qt_press}) - (\ref{mom._1_qt_press}). Given $\nu>\kappa$, let $\mu=\nu-\sqrt{\nu^{2}-\kappa^{2}}$, $\lambda=\sqrt{\nu^{2}-\kappa^{2}}$ and $w_{n}=u_{n}+\mu\nabla\log\varrho_{n}$. Then, $\left(\varrho_{n},w_{n}\right)$ satisfies
	\begin{equation} \label{cont_w}
	\partial_{t}\varrho_{n}+\textrm{div}\left(\varrho_{n}w_{n}\right)=\mu\Delta\varrho_{n},
	\end{equation}
	\begin{equation} \label{mom_w}
	\partial_{t}\left(\varrho_{n}w_{n}\right)+\textrm{div}\left(\varrho_{n}w_{n}\otimes w_{n}\right)-\mu\Delta\left(\varrho_{n}w_{n}\right)-2\lambda\textrm{div}\left(\varrho_{n}Dw_{n}\right)=0.
	\end{equation}
\end{lemma}
\begin{proof}
	The proof follows the same line as in \cite{AS_1} and it could be easily reproduced without the presence of the pressure contribution. 
\end{proof}

The next lemma concerns the energy estimate to the system (\ref{cont_w}) - (\ref{mom_w}).
\begin{lemma} \label{e_est}
 	Let $\left(\varrho_{n},u_{n}\right)$ be a smooth solution of (\ref{cont._1_qt_press}) - (\ref{mom._1_qt_press}). Then, $w_{n}=u_{n}+\mu\nabla\log\varrho_{n}$ and $\varrho_{n}$ satisfy
 	\begin{equation} \label{ee_w}
 	\frac{d}{dt}\int_{\Omega}\varrho_{n}\frac{\left|w_{n}\right|^{2}}{2}dx+2\lambda\int_{\Omega}\varrho_{n}\left|Dw_{n}\right|^{2}dx+\mu\int_{\Omega}\varrho_{n}\left|\nabla w_{n}\right|^{2}dx=0.
 	\end{equation}
\end{lemma}
\begin{proof}
	In order to get (\ref{ee_w}), we multiply (\ref{mom_w}) by $w_{n}$, we integrate in space and we use (\ref{cont_w}).
\end{proof}
The next \textit{a priori estimate}	is analogous to the one  in \cite{MV}.
\begin{lemma} \label{MV-der}	
 	Let $\left(\varrho_{n},u_{n}\right)$ be a smooth solution of (\ref{cont._1_qt_press}) - (\ref{mom._1_qt_press}). Then, $w_{n}=u_{n}+\mu\nabla\log\varrho_{n}$ and $\varrho_{n}$ satisfy
 	$$
 	\frac{d}{dt}\int_{\Omega}\varrho_{n}\left(1+\frac{\left|w_{n}\right|^{2}}{2}\right)\log\left(1+\frac{\left|w_{n}\right|^{2}}{2}\right)dx+\mu\int_{\Omega}\varrho_{n}\left|\nabla w_{n}\right|^{2}dx
 	$$
 	$$
 	+\mu\int_{\Omega}\varrho_{n}\left|\nabla w_{n}\right|^{2}\log\left(1+\frac{\left|w_{n}\right|^{2}}{2}\right)dx+2\lambda\int_{\Omega}\varrho_{n}\left|Dw_{n}\right|^{2}dx
 	$$
 	$$
 	+2\lambda\int_{\Omega}\varrho_{n}\left|Dw_{n}\right|^{2}\log\left(1+\frac{\left|w_{n}\right|^{2}}{2}\right)dx+\mu\int_{\Omega}\varrho_{n}\left|\nabla\frac{\left|w_{n}\right|^{2}}{2}\right|^{2}\frac{2}{2+\left|w_{n}\right|^{2}}dx
 	$$
 	\begin{equation} \label{mv_qw}
 	+2\lambda\int_{\Omega}\varrho_{n}Dw_{n}w_{n}\nabla w_{n}w_{n}\frac{2}{2+\left|w_{n}\right|^{2}}dx=0.
 	\end{equation}
\end{lemma}
\begin{proof}
	Let $\beta\in C^{1}\left(\mathbb{R}\right)$. By integration by parts, we have
	$$
	-\mu\int_{\Omega}\Delta\left(\varrho_{n}w_{n}\right)w_{n}\beta^{\prime}\left(\frac{\left|w_{n}\right|^{2}}{2}\right)dx
	$$
	$$
	=-\mu\int_{\Omega}\Delta\varrho_{n}\left|w_{n}\right|^{2}\beta^{\prime}\left(\frac{\left|w_{n}\right|^{2}}{2}\right)dx+\mu\int_{\Omega}\Delta\varrho_{n}\beta\left(\frac{\left|w_{n}\right|^{2}}{2}\right)dx
	$$
	\begin{equation} \label{mv_qw_1}
	+\mu\int_{\Omega}\varrho_{n}\left|\nabla w_{n}\right|^{2}\beta^{\prime}\left(\frac{\left|w_{n}\right|^{2}}{2}\right)dx+\mu\int_{\Omega}\varrho_{n}\left|\nabla\frac{\left|w_{n}\right|^{2}}{2}\right|^{2}\beta^{\prime\prime}\left(\frac{\left|w_{n}\right|^{2}}{2}\right)dx.
	\end{equation}
	Now, multiplying (\ref{mom_w}) by $w_{n}\beta^{\prime}\left(\frac{\left|w_{n}\right|^{2}}{2}\right)$ and integrating by parts, we obtain
	$$
	\frac{d}{dt}\int_{\Omega}\varrho_{n}\beta\left(\frac{\left|w_{n}\right|^{2}}{2}\right)dx+\mu\int_{\Omega}\varrho_{n}\left|\nabla w_{n}\right|^{2}\beta^{\prime}\left(\frac{\left|w_{n}\right|^{2}}{2}\right)dx
	$$
	\begin{equation}
	+\mu\int_{\Omega}\varrho_{n}\left|\nabla\frac{\left|w_{n}\right|^{2}}{2}\right|^{2}\beta^{\prime\prime}\left(\frac{\left|w_{n}\right|^{2}}{2}\right)dx-2\lambda\int_{\Omega}\textrm{div}\left(\varrho_{n}Dw_{n}\right)w_{n}\beta^{\prime}\left(\frac{\left|w_{n}\right|^{2}}{2}\right)dx.
	\end{equation}
	We integrate by parts the last term, namely
	$$
	-2\lambda\int_{\Omega}\textrm{div}\left(\varrho_{n}Dw_{n}\right)w_{n}\beta^{\prime}\left(\frac{\left|w_{n}\right|^{2}}{2}\right)dx
	$$
	$$
	=2\lambda\int_{\Omega}\varrho_{n}\left|Dw_{n}\right|^{2}\beta^{\prime}\left(\frac{\left|w_{n}\right|^{2}}{2}\right)dx+2\lambda\int_{\Omega}\varrho_{n}Dw_{n}w_{n}\nabla w_{n}w_{n}\beta^{{\prime}\prime}\left(\frac{\left|w_{n}\right|^{2}}{2}\right)dx,
	$$
	and choosing $\beta\left(t\right)=\left(1+t\right)\log\left(1+t\right)$ we get (\ref{mv_qw}).
\end{proof}

Now, thanks to Lemmas \ref{ei_press} and  \ref{e_est}, we have the following uniform bounds
$$
\left\Vert \sqrt{\varrho_{n}}u_{n}\right\Vert _{L^{\infty}\left(0,T;L^{2}\left(\Omega\right)\right)}\leq C, \ \ \left\Vert \nabla\sqrt{\varrho_{n}}\right\Vert _{L^{\infty}\left(0,T;L^{2}\left(\Omega\right)\right)}\leq C,
$$
$$
\left\Vert \varrho_{n}\right\Vert _{L^{\infty}\left(0,T;L^{1}\left(\Omega\right)\right)}\leq C, \ \ \left\Vert \sqrt{\varrho_{n}}\nabla u_{n}\right\Vert _{L^{2}\left(0,T;L^{2}\left(\Omega\right)\right)}\leq C,
$$
\begin{equation} \label{ub}
\left\Vert \sqrt{\varrho_{n}}w_{n}\right\Vert _{L^{\infty}\left(0,T;L^{2}\left(\Omega\right)\right)}\leq C, \ \ \left\Vert \sqrt{\varrho_{n}}\nabla w_{n}\right\Vert _{L^{2}\left(0,T;L^{2}\left(\Omega\right)\right)}\leq C.
\end{equation}
Moreover, the following bound holds (see Antonelli and Spirito \cite{AS_1})
\begin{equation} \label{rad_rho}
\left\Vert \sqrt{\varrho_{n}}\right\Vert _{L^{2}\left(0,T;H^{2}\left(\Omega\right)\right)}\leq C.
\end{equation}
The above bounds and the Lemma \ref{MV-der} allow to prove a Mellet-Vasseur type estimate for $\left(\varrho_{n},w_{n}\right)$, namely, under suitable assumptions on the initial data, the following estimate holds (see Antonelli and Spirito \cite{AS_1}, Lemma 4.2)
\begin{equation} \label{MV-est}
\sup_{t}\int_{\Omega}\varrho_{n}\left(1+\frac{\left|w_{n}\right|^{2}}{2}\right)\log\left(1+\frac{\left|w_{n}\right|^{2}}{2}\right)dx\leq C.
\end{equation}
The above estimate is crucial in order to prove the convergence result. In particular, we have the following lemma
\begin{lemma} \label{con}	
	Let $\left(\varrho_{n},u_{n}\right)_{n \in N}$ be a sequence of solutions of (\ref{cont._1_qt}) - (\ref{mom._1_qt}) and $w_{n}=u_{n}+\mu\nabla\log\varrho_{n}$. Then, up to subsequences, we have
	\begin{equation} \label{conv}
	\sqrt{\varrho_{n}}\rightarrow\sqrt{\varrho}\mbox{ in }L^{2}\left(0,T;H^{1}\left(\Omega\right)\right), \ \ \sqrt{\varrho_{n}}u_{n}\rightarrow\sqrt{\varrho}u\mbox{ in }L^{2}\left(0,T;L^{2}\left(\Omega\right)\right),
	\end{equation}
	as $\varepsilon\rightarrow0$. Moreover, $(\varrho,u)$ is a weak solution of the system (\ref{cont._1_qt_press}) - (\ref{mom._1_qt_press}).
\end{lemma} 	 	
\begin{proof}
	The proof follows the same line as in \cite{AS_1}, see Lemma 4.3 and 4.4,  and it could be reproduced dropping the pressure contribution. Indeed, similarly to the capillary case, for $\varepsilon>0$ the pressure contribution appears in the estimate (4.12) of Lemma 4.3. This means that the bound independent by $\varepsilon$ of the quantity $\varrho_{\varepsilon}^{\gamma}$ (see \ref{convergence-1}) is also required in order to take the $\varepsilon$-limit in (4.12). Lemma 4.4 does not require any analysis of the pressure contribution.
	
\end{proof}


\bigskip
\begin{thebibliography}{10}
	
\bibitem{AS_3} Antonelli P.,  Spirito S., Global existence of weak solutions to the Navier-Stokes-Korteweg equations, arXiv:1903.02441.
	
\bibitem{AS_2} Antonelli P., Spirito S., On the compactness of weak solutions to the Navier-Stokes-Korteweg equations for capillary fluids, \textit{Nonlinear Analysis}, \textbf{187}, 110--124, 2019.
	
\bibitem{AS_1} Antonelli P., Spirito S., On the compactness of finite energy weak solutions to the quantum Navier-Stokes equations, \textit{J. Hyperbolic Differ. Equ.}, \textbf{15(1)}, 133--147, 2018.
	
\bibitem{AS} Antonelli P., Spirito S., Global Existence of Finite Energy Weak Solutions of Quantum Navier--Stokes Equations, \textit{Arch. Rational Mech. Anal.}, \textbf{225}, 1161--1199, 2017.
	
\bibitem{BG} Benzoni--Gavage S., Propagating phase boundaries and capillary fluids, lecture notes for the International Summer School on "Mathematical Fluid Dynamics", Levico Terme (Trento), 2010.
	
\bibitem{BDL_0} Bresch D., B. Desjardins, Sur un mod\' ele de Saint-
Venant visqueux et sa limite quasi-g\' eostrophique, \textit{C. R. Math. Acad. Sci. Paris}, \textbf{35(12)}, 1079--1084, 2002.
	
\bibitem{BDL} Bresch D., Desjardins  b., and  Lin C. H., On some compressible fluid models: Korteweg, lubrication, and shallow water systems, \textit{Comm. Partial Differential Equations}, \textbf{28(3-4)}, 843--886, 2003.
%
%
%
%
	
\bibitem{BGL} Bresch D., Gisclon M.  and  Lacroix-Violet I., On Navier-Stokes-Korteweg and Euler-Korteweg Systems: Application to Quantum Fluids Models, \textit{Arch. Rational Mech. Anal.}, \textbf{233(3)}, 975--1025, 2019.
	
	
\bibitem{BNV_1} Bresch D.,  Pascal N. and Vila J.-P., Relative entropy for compressible Navier-Stokes equations with density-dependent viscosities and applications, \textit{C. R. Math. Acad. Sci. Paris}, \textbf{354(1)}, 45--49, 2016.
	
\bibitem{BFPV} Bresch D., C. Fr\' ed\' eric, Pascal N. and  Vila J. -P., A generalization of the quantum Bohm identity: hyperbolic CFL condition for Euler-Korteweg equations, \textit{C. R. Math. Acad. Sci. Paris}, \textbf{354(1)}, 39--43, 2016.
	
\bibitem{BD-4} Bresch D.,  Desjardins B. and Zatorska~ E., Two-velocity hydrodynamics in fluid mechanics: Part II. Existence of global $\kappa$-entropy solutions to the compressible Navier-Stokes systems with degenerate viscosities, \textit{J. Math. Pure. Appl.}, \textbf{104}, 801--836, 2015.
	
\bibitem{BD-5} Bresch D., Vasseur~A. and Yu~C., Global existence of entropy-weak solutions to the compressible Navier-Stokes equations with non-linear density dependent viscosities, arXiv:1905.02701v1, 2019.
	
\bibitem{CN17}
Caggio~M. , Necasova~S.,  Inviscid incompressible limits for rotating fluids. Nonlinear Anal. \textbf{163}, 1--18, 2017.
	
\bibitem{CDM13} 
Chen~L., Donatelli ~D. and Marcati~P., 
Incompressible type limit analysis of a hydrodynamic model for charge-carrier transport.  \textit{SIAM J. Math. Anal}. {\bf 45 (3)}, 915 - 933, 2013.
	
\bibitem{DG99}
Desjardins B.~ and Grenier~E., Low {M}ach number limit of viscous
  compressible flows in the whole space,  \textit{R. Soc. Lond. Proc. Ser. A Math.
  Phys. Eng. Sci}. \textbf{455(1986)},  2271--2279, 1999.


	
\bibitem{DE} Desjardins B. and M. J. Esteban, On weak solutions for fluid--rigid structure interaction: compressible and incompressible models, \textit{Comm. Partial Diff. Eqs.}, \textbf{25(7--8)}, 1399--1413, 2000.
	
\bibitem{DFN10}
Donatelli~D., Feireisl~E., and Novotn{\'y}~A.,
On incompressible limits for the {N}avier-{S}tokes system on unbounded domains
  under slip boundary conditions,   \textit{Discrete Contin. Dyn. Syst. Ser. B}
  \textbf{13(4)}, 783--798, 2010. 
	
	
\bibitem{DEP} Donatelli D., Feireisl ~E. and  Marcati~P., Well/ill posedness for the Euler-Korteweg-Poisson system and related problems, \textit{Comm. Partial Diff. Eqs.}, \textbf{40(7)}, 1314--1335, 2015.
	
\bibitem{DM12a}
 Donatelli~D. and Marcati~P.,  Low Mach number limit on exterior domains,   \textit{Acta Mathematica Scientia}, \textbf{32(1)}, 164--176, 2012. 

\bibitem{DM16}
 Donatelli~D. and Marcati~P., Low Mach number limit for the quantum hydrodynamics system,  \textit{Res. Math. Sci.} 3 (2016), Paper No. 13, 13 pp.
	
\bibitem{DBN18}
D.~ Donatelli, B~Ducomet and S.~Necasova,  Low Mach number limit for a model of accretion disk,  \textit{Discrete Contin. Dyn. Syst}. \textbf{38(7)}, 3239--3268, 2018. 	

 \bibitem{Donatelli-Trivisa-2008}
Donatelli~D., Trivisa~K., 
From the dynamics of gaseous stars to the incompressible Euler equations,  \textit{J. Differential Equations}, \textbf{245},  1356--1385, 2008.


	
\bibitem{Do} Dong J., A note on barotropic compressible quantum navier-Stokes equations, \textit{Nonlinear Analysis: Real World Applications}, \textbf{73}, 854--856, 2010.
	
\bibitem{E} Evans L. C., Partial Differential Equations, \textit{Amer. Math. Soc., Providence, RI}, 1993.
	
\bibitem{FSN} Feireisl E., A. Novotn\' y and Y. Sun, Yongzhong, Suitable weak solutions to the Navier-Stokes equations of compressible viscous fluids, \textit{Indiana Univ. Math. J.}, \textbf{60(2)}, 611--631, 2011.
	
\bibitem{FJN} Feireisl E., B. J. Jin and A. Novotn\' y, Relative Entropies, Suitable Weak Solutions, and Weak-Strong Uniqueness for the Compressible Navier-Stokes System, \textit{J. Math. Fluid Mech.}, \textbf{14}, 717--730, 2012.
	
	
\bibitem{GaVa} Galaktionov V. A. and J. L. V\' azquez, A stability Technique for Evolution Partial Differential Equations, Birkh\" auser, 2004.
	
\bibitem{GLV} Gisclon M. and I. Lacroix-Violet, About the barotropic compressible quantum Navier-Stokes equations, \textit{Nonlinear Analysis}, \textbf{128}, 106--121, 2015.
	
	
\bibitem{BH-1} Haspot B., From the highly compressible Navier--Stokes equations to fast diffusion and porous media equations, existence of global weak solution for the quasi-solutions, \textit{J. Math. Fluid Mech.}, \textbf{18}, 243--291, 2016.
	
\bibitem{HZ} Haspot B. and  Zatorska E., From the highly compressible Navier--Stokes equations to the porous medium equation--rate of convergence, \textit{Discrete Contin. Dyn. Syst.}, \textbf{A 36}, 3107--3123, 2016.
	
	
	
	
	
\bibitem{Ji} Jiang F. A remark on weak solutions to the barotropic compressible quantum Navier-Stokes equations, \textit{Nonlinear Analysis: Real World Applications}, \textbf{12}, 1733--1735, 2011.
	
\bibitem{J} J\"ungel A., Global weak solutions to compressible Navier--Stokes equations for quantum
fluids, \textit{SIAM J. Math. Anal.}, \textbf{42(3)}, 1025--1045, 2010.
	
	
		
\bibitem{LX} Li J. and Xin~Z., Global existence of weak solutions to the barotropic compressible Navier--Stokes flows with degenerate viscosities (http://arxiv.org/abs/1504.06826v2).
	
\bibitem{L} Liang Z., Vanishing pressure limit for compressible
Navier--Stokes equations with degenerate viscosities, \textit{Nonlinearity}, \textbf{30}, 4173--4190, 2017.
	
	
\bibitem{L-P.L.M98}
 Lions~P.-L.,  and Masmoudi~N., Incompressible limit for a viscous
  compressible fluid,  \textit{J. Math. Pures Appl}. (9) \textbf{77}, 585--627, 1998.
  	\bibitem{MV} Mellet A. and A. Vasseur, On the barotropic compressible Navier--Stokes equations, \textit{Commun. PDE}, \textbf{32}, 431--452, 2007.
	
	
\bibitem{VY} Vasseur A. and Yu C., Existence of global weak solutions for 3D degenerate compressible Navier--Stokes equations, \textit{Invent. Math.}, \textbf{206}, 935--974, 2016.
	
\bibitem{VY-1} Vasseur A. and  Yu~C, Global weak solutions to compressible quantum Navier-Stokes equations with damping, \textit{SIAM J. Math. Anal.}, \textbf{48(2)}, 1489--1511, 2016.
	
	
\end{thebibliography}
\end{document}